\documentclass{amsart}

\usepackage{amssymb}
\usepackage{amsmath}
\usepackage{appendix}
\usepackage{nicefrac}
\usepackage{pstricks}
\usepackage{graphicx}
\usepackage{url}
\usepackage{pst-pdf}

\usepackage[cmtip,all]{xy}

\newtheorem{theorem}{Theorem}[section]

\newtheorem{lemma}[theorem]{Lemma}
\newtheorem{proposition}[theorem]{Proposition}

\theoremstyle{definition}
\newtheorem{definition}[theorem]{Definition}
\newtheorem{example}[theorem]{Example}

\theoremstyle{remark}
\newtheorem{remark}[theorem]{Remark}

\numberwithin{equation}{section}
\usepackage{todonotes}
\usepackage{xcolor}
\usepackage{pstricks}
\usepackage{graphicx}

\begin{document}

\title{Universal bounds on the entropy of toroidal attractors}


\author{P. Montealegre Mac\'ias}
\address{Facultad de Ciencias Matem\'aticas, Universidad Complutense de Madrid, 28040 Madrid, Espa\~na}
\curraddr{}
\email{pemont02@ucm.es}

\author{J.J. S\'anchez-Gabites}
\address{Facultad de Ciencias Matem\'aticas, Universidad Complutense de Madrid, 28040 Madrid, Espa\~na}
\curraddr{}
\email{jajsanch@ucm.es}
\thanks{The second author is funded by the Spanish Ministerio de Ciencia e Innovaci\'on through grant PGC2018-098321-B-I00 and the Ram\'on y Cajal programme (RYC2018-025843-I)}

\date{}

\begin{abstract} A toroidal set is a compactum $K \subseteq \mathbb{R}^3$ which has a neighbourhood basis of solid tori. We study the topological entropy of toroidal attractors $K$, bounding it from below in terms of purely topological properties of $K$. In particular, we show that for a toroidal set $K$, either any smooth attracting dynamics on $K$ has an entropy at least $\log 2$, or (up to continuation) $K$ admits smooth attracting dynamics which are stationary (hence with a zero entropy).
\end{abstract}

\maketitle

\section{Introduction}

 It is well known that attractors can be very complicated both topologically and dynamically. One wonders to what extent these two sorts of complexity are related, and in particular whether ``topological strangeness'' of a compactum $K$ alone may already force a certain degree of complexity on \emph{any} possible attracting dynamics on $K$. With the usual interpretation that a strictly positive entropy $h$ is indicative of complicated dynamics, we are therefore interested in bounds of the form \begin{equation} \label{eq:bound} h(f|_K) \geq \log p(K) > 0\end{equation} where $f$ is any $\mathcal{C}^{\infty}$ diffeomorphism having the compactum $K$ as an attractor and, crucially, $p(K)$ depends \emph{only} on $K$ and \emph{not} on $f$. In other words, the bound $\log p(K)$ is universal across all $\mathcal{C}^{\infty}$ attracting dynamics on a given set $K$.

We shall argue (Subsection \ref{sub:why}) that, in a certain sense, the simplest class of compacta $K$ for which universal bounds of the form \eqref{eq:bound} are possible is that of \emph{toroidal} sets. These are compacta $K \subseteq \mathbb{R}^3$ which have a neighbourhood basis comprised of solid tori (\cite{hecyo1}). These tori can wind inside each other and be knotted, and so $K$ will usually be a very complicated continuum. Among toroidal sets one finds every smooth or polygonal knot, many wild knots, the usual embeddings of $n$--adic solenoids and generalized solenoids (where $n$ is replaced with a sequence $\{n_k\}$), knotted solenoids, every inverse limit of a self map of $\mathbb{S}^1$ (see \cite[\S3, pp. 375ff.]{kato1}), classical continua such as the Whitehead continuum, etc.

To each toroidal set $K$ (no dynamics yet) one can assign a collection of prime numbers $p_i \geq 2$ called its \emph{prime divisors} (\cite{hecyo2}). These capture purely topological properties of $K$; roughly speaking, they are related to the amount of ``self-winding'' of $K$. A toroidal set may have any number of prime divisors, possibly none or possibly infinite.

In order to analyze dynamics on toroidal sets we shall introduce the notion of the \emph{geometric degree} $d_{\mathcal{N}}(f;K)$ of a local homeomorphism $f$ which leaves a toroidal set $K$ invariant. We shall concentrate almost exclusively on the case when $K$ is an attractor for $f$. Then the geometric degree $d_{\mathcal{N}}(f;K)$ is a positive integer which, very roughly, counts the minimum number of ``angular preimages'' of a point. Unlike the ordinary (homological) degree, where preimages are counted as $\pm 1$ depending on the local behaviour of the map, in the geometric degree each preimage contributes $+1$ to the count. This makes it more sensitive than the ordinary degree while still retaining the usual properties of the latter. In particular it is invariant under homotopies or, more precisely, under continuations in the sense of Conley: if $(f_{\lambda})_{\lambda \in [0,1]}$ is a continuous family of local homeomorphisms, each having a toroidal attractor $K_{\lambda}$ (also varying continuously with $\lambda$), the prime divisors of the $K_{\lambda}$ and the geometric degrees $d_{\mathcal{N}}(f_{\lambda};K_{\lambda})$ are independent of $\lambda$.

The geometric degree provides a link between entropy and prime divisors. On the one hand, if $K$ is an attractor for a local homeomorphism $f$ then the prime divisors of the integer $d_{\mathcal{N}}(f;K)$, in the ordinary sense of arithmetics, coincide with the prime divisors $\{p_i\}$ of the toroidal set $K$ (Theorem \ref{thm:dattract}) and so in particular $d_{\mathcal{N}}(f;K) \geq \prod_i p_i$. On the other, when $f$ is $\mathcal{C}^{\infty}$ the geometric degree provides the lower bound $h(f|_K) \geq \log d_{\mathcal{N}}(f;K)$ for the entropy (Theorem \ref{thm:manning}). Therefore

\begin{equation} \label{eq:intro1} h(f|_K) \geq \log d_{\mathcal{N}}(f;K) \geq \log \prod_i p_i. \end{equation} 

Since the prime divisors $p_i$ depend only on $K$ but not on the dynamics, this bound is of the desired form \eqref{eq:bound} with $p(K) = \prod_i p_i$. In particular as soon as $K$ has at least one prime divisor we get $h(f|_K) \geq \log 2 > 0$. Notice that the bound $\log 2$  is universal not only across all $\mathcal{C}^{\infty}$ attracting dynamics on a given $K$ but in fact across all toroidal attractors with at least one prime divisor. In this sense it is sharp, since it is attained by the standard embedding of the dyadic solenoid in $\mathbb{R}^3$ with its usual dynamics.

If a toroidal attractor $K$ has no prime divisors then \eqref{eq:intro1} reduces to the trivial bound $h(f|_K) \geq \log 1 = 0$ and, given the universal nature of the bound, the (somewhat wild) question arises of whether $K$ actually supports some attracting dynamics with a zero entropy. Since the geometric degree and the prime divisors are invariant under continuation, we may also ask whether $K$ can be continued to an attractor with a zero entropy. It turns out that the answer is in the affirmative: we shall prove that if $K$ is a toroidal attractor with no prime divisors, then it can be continued to a smooth knot $\gamma$ which is an attractor with stationary dynamics (Theorem \ref{thm:continuation}). For this result the use of the geometric degree instead of the ordinary one is crucial.

Summing up, up to continuation and for smooth dynamics, we obtain a neat alternative:

\begin{theorem} \label{thm:intro} Let $K$ be a toroidal attractor for a $\mathcal{C}^{\infty}$ diffeomorphism of $\mathbb{R}^3$.
\begin{itemize}
    \item[(i)] If $K$ has some prime divisor, then $K$ and all of its $\mathcal{C}^{\infty}$ toroidal continuations have an entropy at least $\log 2$.
    \item[(ii)] If $K$ has no prime divisors, then $K$ has a $\mathcal{C}^{\infty}$ continuation to a smooth knot which is an attractor with stationary dynamics (hence with a zero entropy).
\end{itemize}
\end{theorem}


We finally discuss briefly the outline of the paper. In Section \ref{sec:background} we recall some definitions and in particular that of the geometric index of a solid torus inside another. This was introduced by Schubert (\cite{schubert}) to study knots and is basic for our constructions. In Section \ref{sec:degree} we define the prime divisors of a toroidal set $K$ and the geometric degree $d_{\mathcal{N}}(f;K)$, and relate the two when $K$ is an attractor. We also prove that they are invariant under continuation. In Section \ref{sec:noprimedivisors} we show that a toroidal attractor $K$ with no prime divisors can be continued to a smooth knot. Section \ref{sec:entropytoroidal} relates the entropy of a toroidal attractor to the geometric degree and hence also to the prime divisors of the attractor. It is here where smoothness assumptions on the dynamics are needed for the first time in order to apply a bound of Yomdin on the entropy. Section \ref{sec:concluding} contains a comparison between the geometric and the ordinary degree as well as an open question related to the smoothness assumption on the dynamics.

We have included two appendices. Appendix \ref{app:index} contains some technicalities about the definition of the geometric index. Appendix \ref{app:yomdin} is a brief outline of Yomdin's bound on entropy, tailored to the very particular case we use in this paper.

\section{Some background} \label{sec:background}

This section gathers some definitions and results that are needed for the rest of the paper. The most important are the definition and properties of the geometric index in Subsection \ref{subsec:geometricindex}.

\subsection{Tame solid tori} \label{subsec:tameness} A solid torus $T$ is a topological space homeomorphic to $\mathbb{D}^2 \times \mathbb{S}^1$; any homeomorphism $h : \mathbb{D}^2 \times \mathbb{S}^1 \longrightarrow T$ is called a framing of $T$. The image under $h$ of $\mathbb{D}^2 \times \{*\}$ is called a meridional disk of $T$ and its boundary (which is a simple closed curve in $\partial T$) a meridian of $T$. The image under $h$ of the curve $\{0\} \times \mathbb{S}^1$ is a called a core of $T$. One should bear in mind that meridional disks can lie in a very crooked way inside $T$; picturing them as being ``radial'' is misleading.

In dimensions three (and higher) there exist wild tori, and these are both inconvenient and unnecessary for our purposes. Henceforth we will almost always confine ourselves to working with tame tori only. Recall that a compact subset $L \subseteq \mathbb{R}^3$ is tame if there exists a homeomorphism $g : \mathbb{R}^3 \longrightarrow \mathbb{R}^3$ which sends $L$ onto a polyhedron, and semilocally tame if there exist an open neighbourhood $U$ of $L$ and a homeomorphism $g : U \longrightarrow g(U) \subseteq \mathbb{R}^3$ which sends $L$ onto a polyhedron. Evidently a tame set $L$ is also semilocally tame, and a deep theorem of Moise states that the converse is also true (\cite[Theorem 3, p. 254]{moise2}).

 In the sequel we will make use of the following facts:

(i) Any solid torus $T \subseteq \mathbb{R}^3$ can be perturbed an arbitrarily small quantity to make it tame. To prove this one picks a framing $h : \mathbb{D}^2 \times \mathbb{S}^1 \longrightarrow T \subseteq \mathbb{R}^3$ and resorts to a classical approximation theorem (see for example \cite[Theorem 2, p. 251]{moise2}) which ensures that for any $\epsilon > 0$ there is another embedding $h' : \mathbb{D}^2 \times \mathbb{S}^1 \longrightarrow \mathbb{R}^3$ that is $\epsilon$--close to $h$ and is piecewise linear, so in particular its image $T'$ is a polyhedral solid torus $\epsilon$--close to $T$.

(ii) Suppose $f$ is a homeomorphism of $\mathbb{R}^3$ or, more generally, a local homeomorphism of $\mathbb{R}^3$, i.e. a homeomorphism $f : U \longrightarrow f(U) \subseteq \mathbb{R}^3$ where $U$ is open in $\mathbb{R}^3$. If $T \subseteq U$ is a tame solid torus, then $f(T)$ is also tame. This follows because $f(T)$ is evidently semilocally tame and, as mentioned above, it is therefore tame.

\subsection{The geometric index} \label{subsec:geometricindex}  

We recall the notion of the geometric index of a solid torus inside another one. This was introduced by Schubert \cite[\S 9]{schubert} (who called it ``order'') and will be essential in the paper. The original definition and results by Schubert are set up for polyhedral tori, but we need a purely topological version. This is not entirely straightforward because of the existence of wild objects in $3$ dimensions, and sorting this out requires some nontrivial results from geometric topology. We have deferred the technical details to Appendix \ref{app:index} since the geometric content of the definition is very intuitive.

We first motivate the definition informally. Let $T_0$ be a solid torus containing a simple closed curve $\alpha$ in its interior. Let $\gamma$ be a core of $T_0$, so that $\alpha$ is homologous to a multiple of $\gamma$; say $\alpha = m \gamma$ with $m \in \mathbb{Z}$. The integer $m$ can be computed (up to a sign) by counting the number of times that $\alpha$ intersects any meridional disk $D$ of $T_0$. This has to be performed algebraically: each point in $D \cap \alpha$ contributes $\pm 1$ depending on the sense in which $\alpha$ crosses $D$ at that point. This algebraic count is independent of the disk $D$ used to perform it. We shall call it the homological winding number of $\alpha$ in $T_0$ and denote it by $m(\alpha \subseteq T_0)$.

The geometric index of $\alpha$ inside $T_0$ is defined in a similar manner but counting the points of intersection in $D \cap \alpha$ geometrically; i.e. each contributes a $+1$ so the count is just the cardinality $|D \cap \alpha|$. This depends on the disk $D$, and one defines the geometric index $N(\alpha \subseteq T_0)$ by minimizing $|D \cap \alpha|$ over all meridional disks $D$. It is clear that $N(\alpha \subseteq T_0) \geq |m(\alpha \subseteq T_0)|$, and the inequality can be strict. For example the Whitehead curve shown in the left panel of Figure \ref{fig:whitehead} has $m=0$ but geometric index $2$.

We now proceed to the formal definitions. For our purposes it is more convenient to define the geometric index of a solid torus (rather than a simple closed curve) inside another one, but the geometric motivation for the definition is the one just outlined.

Let $T_0 \subseteq \mathbb{R}^3$ be a solid torus and $T_1$ another solid torus contained in the interior of $T_0$. We assume these to be tame. Let $D$ be a meridional disk of $T_0$. We say that $D$ is transverse to $T_1$ if there exist a neighbourhood $N$ of $D$ in $T_0$ and a homeomorphism $h: (N,N \cap T_1) \longrightarrow (\mathbb{D}^2, D_1 \cup \ldots \cup D_r) \times [-1,1]$ where:
\begin{itemize}
    \item[(i)] the $D_i$ are pairwise disjoint closed disks.
    \item[(ii)] $h$ carries $D$ onto $\mathbb{D}^2 \times \{0\}$.
\end{itemize}

This definition is intended to provide a local model for the intersection $D \cap T_1$. Observe that (ii) implies that $D \cap T_1$ consists of $r$ disjoint closed disks (the preimages of the $D_i \times \{0\}$ under $h$). These are \emph{not} assumed to be meridional disks of $T_1$, in contrast to the next definition.

 \begin{definition} \label{defn:Ntop} The geometric index of $T_1$ in $T_0$ is the minimum $r$ such that there exists a meridional disk $D$ of $T_0$ which is transverse to $T_1$ and such that $D \cap T_1$ consists of $r$ meridional disks of $T_1$. We denote this number by $N(T_1 \subseteq T_0)$.
\end{definition}

At this stage it is not even clear that $N(T_1 \subseteq T_0)$ be well defined in general, since perhaps no meridional disk transverse to $T_1$ exists at all. However, for tame tori as we are considering the geometric index is indeed well defined. Moreover, it has the following three fundamental properties: \label{pg:properties}

(P1) If $(T_0, T_1)$ and $(T'_0, T'_1)$ are homeomorphic pairs, then $N(T_1 \subseteq T_0) = N(T'_1 \subseteq T'_0)$.

(P2) $N(T_1 \subseteq T_0) = 0$ if and only if there exists a tame ball $B$ such that $T_1 \subseteq B \subseteq T_0$.

(P3) The geometric index is multiplicative: if $T_2 \subseteq T_1 \subseteq T_0$, then \[N(T_2 \subseteq T_0) = N(T_2 \subseteq T_1) \cdot N(T_1 \subseteq T_0).\]

The first property is direct from the definition, since the notions of meridional disk and transversality are preserved by homeomorphisms. Proving the other two involves some woodworking which is best done when the tori are polyhedral. They are proved, under this assumption, by Schubert (\cite[Hilfssatz 3 and Satz 3, pp. 171 and 175]{schubert}). We shall show in Appendix \ref{app:index} how to translate those results to our topological setting. Here we discuss briefly (P2) because it is quite plausible and also provides a good illustration of why the geometric index is more powerful than the winding number $m$. If the geometric index $N(T_1 \subseteq T_0)$ is zero, there exists a meridional disk of $T_0$ which does not meet $T_1$. Then cutting $T_0$ along this meridional disk produces a ball $B$ between $T_1$ and $T_0$. Conversely, should such a ball $B$ exist, one can shrink it inwards by an ambient isotopy $G_t$ of $T_0$ to a tiny size and find a meridional disk $D$ disjoint from $G_1(B)$. Running the isotopy in the reverse produces a meridional disk disjoint from $B$ and hence from $T_1$. Notice that (P2) is certainly not true for the winding number $m$, and again the Whitehead curve of Figure \ref{fig:whitehead} is a classical example.

We finish by returning to the index $N(\alpha \subseteq T_0)$ of a simple closed curve $\alpha$ inside a solid torus $T_0$. This can be defined just as $N(T_1 \subseteq T_0)$ with the obvious adaptations: in the local model of transversality one replaces the disks $D_i$ with points $p_i$ and in Definition \ref{defn:Ntop} one removes the condition that the $D_i$ be meridional disks of $T_0$. The resulting definition agrees with the informal one given at the beginning of this section. It can be shown that $N(\gamma \subseteq T_1) = N(T_0 \subseteq T_1)$ whenever $\gamma$ is a core of $T_0$. (In fact in Schubert's original work the geometric index is defined first for curves, then for solid tori through this relation).

\subsection{Dynamics} The following definitions are standard and can be found, for example, in \cite{franksricheson} or \cite{katokhasselblatt}.
\medskip

{\it Attractors.} Suppose $f$ is a homeomorphism. A compact set $N$ such that $f(N) \subseteq {\rm int}\ N$ is called a trapping region for $f$. The attractor defined by that trapping region is the set $K := \bigcap_{n \geq 0} f^n(N)$.

A set $P$ is positively invariant if $f(P) \subseteq P$ and invariant if $f(P) = P$. An attractor $K$ has a neighbourhood basis of compact positively invariant sets; namely the $\{f^n(N)\}_{n \geq 0}$. The attractor $K$ itself is invariant.

If the forward orbit of a point $x$ enters $N$ then it converges to $K$ asymptotically, this meaning that for every neighbourhood $V$ of $K$ there exists $n_0$ such that $f^n(x) \in V$ for $n \geq n_0$. The set of points with this property is called the basin of attraction of $K$ and is denoted by $\mathcal{A}(K)$. It is an open, invariant neighbourhood of $K$. One can easily check that in fact not only points in $\mathcal{A}(K)$ are attracted by $K$, but also compact sets as well. Explicitly: for every neighbourhood $V$ of $K$ and every compact set $C \subseteq \mathcal{A}(K)$ there exists $n_0$ such that $f^n(C) \subseteq V$ for all $n \geq n_0$.

An attractor can also be defined intrinsically (i.e. without reference to a trapping region) as a compact invariant set which attracts nearby points and has a neighbourhood basis of positively invariant sets. In turn, the latter condition can be replaced by requiring that $K$ be stable in the sense of Lyapunov. Thus the type of attractors we are considering are sometimes called stable attractors.

In this paper we will consider dynamics generated by a local homeomorphism of $\mathbb{R}^3$. By this we mean a map $f$ defined on an open subset $U \subseteq \mathbb{R}^3$ and which is a homeomorphism onto its image $f(U) \subseteq \mathbb{R}^3$. By the invariance of domain theorem it suffices to require that $f$ be continuous and injective. The definition of an attractor given above generalizes immediately to this situation.
\medskip

 {\it Continuations.} We now recall the notion of a ``continuation'' introduced by Conley \cite{conley1}, tayloring the definition to our particularly simple case of attracting dynamics.

Suppose we have a parametrized family of local homeomorphisms $f_{\lambda}$ all defined on some open set $U \subseteq \mathbb{R}^3$. Here $\lambda$ ranges in some set of parameters which for definiteness we shall take to be the interval $[0,1] \subseteq \mathbb{R}$. Formally, we consider a continuous map $f : [0,1] \times U \longrightarrow \mathbb{R}^3$ such that every partial mapping $f_{\lambda} : U \longrightarrow \mathbb{R}^3$ defined by $f_{\lambda}(x) := f(\lambda,x)$ is injective (hence a local homeomorphism).

Now let $[0,1] \ni \lambda \longmapsto K_{\lambda}$ be a map where each $K_{\lambda}$ is an attractor for $f_{\lambda}$. We want to provide a reasonable definition of ``continuity'' of this map. Fix some $\lambda_0 \in [0,1]$ and let $N$ be a trapping region for $K_{\lambda_0}$, so that $f_{\lambda_0}(N) \subseteq {\rm int}\ N$. Evidently the latter inclusion still holds for $\lambda$ close enough to $\lambda_0$, and so $N$ contains an attractor for $f_{\lambda}$. We say that $\lambda \longmapsto K_{\lambda}$ is continuous at $\lambda_0$ if there is a neighbourhood $I \subseteq [0,1]$ of $\lambda_0$ such that for every $\lambda \in I$ the attractor $K_{\lambda}$ is precisely the one determined by the trapping region $N$ under the dynamics $f_{\lambda}$; that is, $K_{\lambda} := \bigcap_{n \geq 0} f_{\lambda}^n(N)$. It is easy to check that this condition is independent of $N$, although $I$ will usually depend on it. Of course, we say that $\lambda \longmapsto K_{\lambda}$ is continuous if it is continuous at every $\lambda_0 \in [0,1]$. This \emph{ad hoc} definition should be reasonable on intuitive grounds and will suffice for our purposes, but in fact one can set up a topology in the collection of attractors of the $f_{\lambda}$ so that this map is continuous in the ordinary sense of Topology.

A continuous mapping as just described is called a continuation, or a continuation from $K_{\lambda =0}$ to $K_{\lambda = 1}$. Each intermediate $K_{\lambda}$ is also called a continuation of $K_{\lambda =0}$.
\medskip

{\it Entropy.} We also recall the definition of ``topological entropy''. For this purpose, we consider a compact metric space, $(X,\operatorname{d})$, and a continuous map $g:X\longrightarrow X$.

For each $n\ge 0$ we define the metric
\begin{equation*}
    \operatorname{d}_n^g(x,y):=\max_{0\le i\le n}\{\operatorname{d}(g^i(x),g^i(y))\}
\end{equation*}
and, for a given $x\in X$ and the open ball $B(x,\epsilon):=\{y\in X\ |\ \operatorname{d}(x,y)<\varepsilon\}$, the $(n,\epsilon)$-dynamical ball
\begin{equation*}
    B_g(x,\epsilon,n)=\bigcap_{i=0}^n g^{-i}(B(x,\epsilon)).
\end{equation*}

The continuity of $g$ ensures that $B_g(x,n,\epsilon)$ is open. Since $X$ is compact, for any $\epsilon > 0$ there is a minimum number of $(n,\epsilon)$-dynamical balls needed to cover $X$. We denote this number by $S_{\operatorname{d}}(g,n,\epsilon)$. It can be interpreted as the minimum amount of initial conditions which, at a scale $\epsilon$ and up to time $n$, are representative of all possible trajectories of the system.

Now, the topological entropy of $g$ is defined as the following double limit:
\begin{equation*}
    h_{\operatorname{d}}(g):=\lim_{\varepsilon\rightarrow 0}\limsup_{n} \frac{1}{n} \log {S_{\operatorname{d}}(g,n,\varepsilon)}.
\end{equation*}

Although the definition of $h_{\operatorname{d}}$ involves the distance $\operatorname{d}$, it is actually invariant among equivalent distances which define the same topology. This can be proved explicitly or by means of a purely topological definition of the entropy that does not involve distances (as it was defined originally). Henceforth we will omit the distance $\operatorname{d}$ from the notation.

\subsection{Why toroidal sets?} \label{sub:why} It was mentioned in the Introduction that toroidal sets are the simplest for which universal bounds of the form \eqref{eq:bound} are possible. The following example justifies this. For part (iv), recall that a compactum $K \subseteq \mathbb{R}^n$ is called cellular if it has a neighbourhood basis comprised of cells; that is, of sets homeomorphic to the standard closed $n$--ball in $\mathbb{R}^n$.

\begin{example} \label{ex:omni} Let $K \subseteq \mathbb{R}^n$. If
\begin{itemize}
    \item[(i)] $K$ is an attractor for a flow, or
    \item[(ii)] $K$ is a global attractor, or
    \item[(iii)] $K$ is an attractor in dimension $n \leq 2$, or
    \item[(iv)] $K$ is a cellular set,
\end{itemize}
then $K$ can be realized as an attractor with stationary dynamics and so with a zero entropy.
\end{example}
\begin{proof} (i) One just needs to stop the flow on $K$. This is straightforward to do when the flow is $\mathcal{C}^1$ (by multiplying its vectorfield by a nonnegative function which vanishes precisely on $K$) but also true when the flow is merely continuous; see \cite{hecyo3} and references therein.

(ii) and (iv) Every global attractor is cellular and, conversely, every cellular set can be realized as a (global) attractor of a flow which is stationary on the set (these are results of Garay \cite{garay1}).

(iii) An attractor in $\mathbb{R}^n$ has a finitely generated \v{C}ech cohomology (\cite[Theorem 1, p. 2827]{pacoyo1}) and then in dimensions $n \leq 2$ a result of G\"unther and Segal (\cite[Corollary 3, p. 326]{gunthersegal1}) reduces the situation to that of flows.
\end{proof}

\begin{remark} Regarding smoothness of the dynamics, in case (i) the slowed down flow can evidently be made as smooth as the original flow, and in the remaining cases in dimension $n \neq 4$ the dynamics can be made $\mathcal{C}^{\infty}$. This follows from results by Grayson, Norton and Pugh (\cite{graysonpugh1} and \cite{nortonpugh1}).  
\end{remark}

Notice that the example shows that several well known strange attractors (for instance, the Lorenz or H{\'e}non attractors) can also be realized as attractors with stationary dynamics. Thus, ``dynamical strangeness'' is not always necessary for ``topological strangeness''.

Bearing in mind our goal of obtaining positive bounds on the entropy which are universal across all attracting dynamics on a given compactum $K$, the preceding example justifies that we focus on discrete dynamical systems and on dimension $n$ at least $3$. Also, we need to go beyond cellular sets. If one regards cells as handlebodies of genus zero, a natural next step in complexity consists in considering compacta $K \subseteq \mathbb{R}^3$ that have a neighbourhood basis comprised of handlebodies of genus one; i.e. solid tori. These are precisely toroidal sets.

\subsection{Toroidal sets and toroidal attractors} \label{subsec:basis} A compactum $K \subseteq \mathbb{R}^3$ is toroidal if it is not cellular and has a neighbourhood basis comprised of (not necessarily tame) solid tori $\{T_k\}$. One can always choose the $\{T_k\}$ to satisfy the following conditions:
\begin{itemize}
    \item[(i)] Each $T_k$ is a tame solid torus.
    \item[(ii)] Each $T_{k+1}$ is contained in the interior of $T_k$.
    \item[(iii)] The geometric indices $N(T_{k+1} \subseteq T_k)$ are all nonzero.
\end{itemize}

To show that these bases exist, first start with any neighbourhood basis $\{T_k\}$ of $K$ satisfying (ii). This exists by the definition of a toroidal set. Then one perturbs each $T_k$ to a tame torus $T'_k$ by a perturbation of size $\epsilon_k \longrightarrow 0$ chosen inductively to ensure that the tame $T'_k$ still contains $T_{k+1}$ in its interior and is contained in the interior of $T'_{k-1}$. The new tame tori $\{T'_k\}$ for a neighbourhood basis of $K$ which satisfy (i) and (ii). Finally, condition (iii) can be achieved by discarding  finitely many of the $T'_k$. Indeed, for each $k$ such that $N(T'_{k+1} \subseteq T'_k) = 0$ we have (by property (P2) of the geometric index) a ball $T'_{k+1} \subseteq B_k \subseteq T'_k$. Thus if $N(T'_{k+1} \subseteq T'_k) = 0$ for infinitely many $k$ then we have $\{B_k\}$ a neighbourhood basis for $K$ comprised of balls, so $K$ would be cellular. This contradicts the definition of a toroidal set. Hence $N(T'_{k+1} \subseteq T'_k) = 0$ for finitely many $k$, and so by discarding these we may achieve the three conditions enumerated above. From now on whenever we speak of ``a basis'' of a toroidal set we will always mean a neighbourhood basis that satisfies the conditions enumerated above.

When a toroidal set $K$ is an attractor for a local homeomorphism $f$ there is a natural way of constructing bases $\{T_k\}$. Since $K$ is toroidal and its basin of attraction is an open neighbourhood of $K$, it contains some solid torus which is a neighbourhood of $K$. This compact set is attracted by $K$, and so there exists $n_0$ such that $f^n(T) \subseteq {\rm int}\ T$ for every $n \geq n_0$. Thus $\{T_k\} := \{T, f^{n_0}(T), f^{2n_0}(T), \ldots\}$ is a neighbourhood basis of $K$ comprised of nested solid tori. If $T$ is taken to be tame (this can always be done by perturbing it to a polyhedral torus) then all the iterates $f^{kn_0}(T)$ are semilocally tame and hence tame. Thus $\{T_k\}$ satisfies properties (i) to (iii) listed above. It has the additional crucial property that each pair $(f^{(k+1)n_0}(T),f^{kn_0}(T))$ is homeomorphic to $(T,f^{n_0}(T))$ and so the geometric indices $N(T_{k+1} \subseteq T_k)$ are all equal. We will call a basis constructed in this manner a dynamically generated basis.

\begin{remark} \label{rem:nottori} A toroidal attractor has a neighbourhood basis of solid tori (because it is toroidal) and a neighbourhood basis of positively invariant sets (because it is an attractor), but there is in principle no guarantee that it has a basis of neighbourhoods which satisfy simultaneously both conditions; i.e. which are positively invariant solid tori. We do not know if this is generally true.
\end{remark}

We conclude by observing that not every toroidal set can be realized as an attractor, and characterizing topologically which can is an open problem. For example, (i) any smooth knot in $\mathbb{R}^3$ can be realized as a (toroidal) attractor with stationary dynamics, but (ii) there are toroidal knots which are smooth everywhere except at a single point and cannot be realized as attractors whatsoever; however (iii) there are toroidal knots which are nowhere smooth (in fact, they are wild everywhere in the sense of Geometric Topology) and they can again be realized as attractors with stationary dynamics. Thus a rather natural geometric gradation of complexity (smooth everywhere, smooth but at a single point, smooth nowhere) does not have a consistent dynamical counterpart.

\section{Prime divisors and the geometric degree \label{sec:degree}}

In this section we associate to each toroidal set a collection of prime numbers (possibly empty or infinite) called its prime divisors. They capture purely topological information about the self-winding of $K$. We also associate to each local homeomorphism between two toroidal sets a rational number called its geometric degree. This is then particularized to the case of a toroidal attractor of a local homeomorphism. Some motivation for the definitions to come can be found in Section \ref{sec:concluding}.

\subsection{Prime divisors} Let $K$ be a toroidal set and $\{T_k\}$ a basis for $K$ as described in Subsection \ref{subsec:basis}.

\begin{definition} \label{def:primedivisor} A prime divisor of $K$ is a prime number $p \geq 2$ which divides $N(T_{k+1} \subseteq T_k)$ for infinitely many $k$.
\end{definition}

By multiplicativity of the geometric index and the primality of $p$ this is equivalent to requiring that the number $N(T_k \subseteq T_1)$ contains arbitrarily large powers of $p$ as $k \rightarrow +\infty$. A yet equivalent condition is that for any fixed $k_1$, the numbers $N(T_k \subseteq T_{k_1})$ should contain arbitrarily large powers of $p$ as $k \rightarrow +\infty$.

\begin{proposition} \label{prop:pdivisor} Being a prime divisor is independent of the basis $\{T_k\}$.
\end{proposition}
\begin{proof} Let $\{T_k\}$ and $\{T'_{\ell}\}$ be two bases for $K$ and suppose that $p$ satisfies Definition \ref{def:primedivisor} for $\{T_k\}$. We show that it also satisfies the definition for $\{T'_{\ell}\}$. Pick $k_1$ so that $T_{k_1} \subseteq T'_1$. Now, for a given power $p^n$ of $p$ let $k \geq k_1$ be big enough so that $p^n | N(T_k \subseteq T_{k_1})$. Finally, pick $\ell$ so that $T'_{\ell} \subseteq T_k$. Then by multiplicativity of the geometric index applied to $T'_1 \supseteq T_{k_1} \supseteq T_k \supseteq T'_{\ell}$ we have $N(T_k \subseteq T_{k_1}) | N(T'_{\ell} \subseteq T'_1)$. Thus $p^n | N(T'_{\ell} \subseteq T'_1)$ and so $N(T_{\ell} \subseteq T_1)$ contains arbitrarily large powers of $p$ as $\ell \rightarrow +\infty$, as was to be shown.
\end{proof}

The prime divisors of a toroidal set were first defined in \cite{hecyo2} in a less elementary fashion. The definition given here is equivalent.

\begin{example} \label{ex:basic} (1) Suppose $K \subseteq \mathbb{R}^3$ is a smooth knot. Picking a closed tubular neighbourhood $T$ of $K$ and a diffeomorphism $(T,K) \cong (\mathbb{D}^2 \times \mathbb{S}^1, 0 \times \mathbb{S}^1)$ it is straightforward to construct a neighbourhood basis of $K$ which consists of nested, solid tori such that the geometric index of each consecutive pair is $1$. Thus any smooth knot is a toroidal set with no prime divisors. The same holds for a polygonal knot taking pl regular neighbourhoods instead of tubular neighbourhoods.

(2) Suppose $K$ is the intersection of a nested family of solid tori $\{T_i\}$ such that each $T_{i+1}$ winds monotonically (i.e. without doubling back) $n_i \geq 1$ times inside $T_i$. The monotonicity condition ensures that there exists a meridional disk of $T_i$ which intersects $T_{i+1}$ along $n_i$ disks (so $N(T_{i+1} \subseteq T_i) \leq n_i$) and at each of these intersections $T_{i+1}$ crosses the meridional disk in the same sense (so $m(T_{i+1} \subseteq T_i) = n_i$). Thus $N(T_{i+1} \subseteq T_i) = n_i$, and so the prime divisors of $K$ are exactly those prime numbers that divide infinitely many of the $n_i$. In particular when $n_i = n$ (as in the standard embedding of an $n$--adic solenoid) the prime divisors of $K$ are the prime divisors of $n$ without multiplicity.
\end{example}

A more interesting example is provided by Whitehead continua.

\begin{example} \label{ex:whiteheadtype} Start with an unknotted solid torus $T_1$ and place a thinner one $T_2$ in its interior along the black curve depicted in the left panel of Figure \ref{fig:whitehead}. Then place an even thinner torus $T_3$ inside $T_2$ following the same pattern (by this we mean that $(T_2,T_3)$ is homeomorphic to $(T_1,T_2)$) and so on. This produces a nested sequence of solid tori $\{T_k\}$ whose intersection $K$ is a toroidal set called a Whitehead continuum, an example of which is shown in the right panel of Figure \ref{fig:whitehead}. This prescription only determines the isotopy class of the core of $T_{k+1}$ inside $T_k$, so $K$ is far from being uniquely determined.

To compute the prime divisors of $K$ notice that $N(T_{k+1} \subseteq T_k) = N(T_2 \subseteq T_1)$ by construction (and invariance of the geometric index under homeomorphisms). Thus the prime divisors of $K$ are exactly the prime numbers that divide $N(T_2 \subseteq T_1)$. It is clear from the drawing that $N(T_2 \subseteq T_1) \leq 2$ since there are obvious meridional disks of $T_1$ which intersect $T_2$ in exactly two points. In fact one has $N(T_2 \subseteq T_1) = 2$; i.e. no meridional disk of $T_1$ transverse to $T_2$ intersects it in less than two disks. This seems intuitively reasonable but is not completely trivial to prove. A quick argument goes as follows. Since the geometric index $N(T_2 \subseteq T_1)$ and the homological winding number $m(T_2 \subseteq T_1)$ are both given by counting intersections with a meridional disk, one with a sign and the other without, they must differ in an even number. But $m(T_2 \subseteq T_1)$ is clearly zero, so $N(T_2 \subseteq T_1)$ must be even and hence either $0$ or $2$. The core of $T_2$ is linked with a meridian of $T$ (in the sense of knot theory) and therefore it cannot be contained in a ball in $T_1$, so $N(T_2 \subseteq T_1)$ must be nonzero, showing that $N(T_2 \subseteq T_1) = 2$ indeed. Thus $K$ has exactly the prime divisor $p = 2$.

\begin{figure}[h]
\centering
\begin{pspicture}(0,0)(12,6.5)
\rput[bl](-2.5,-6){\includegraphics{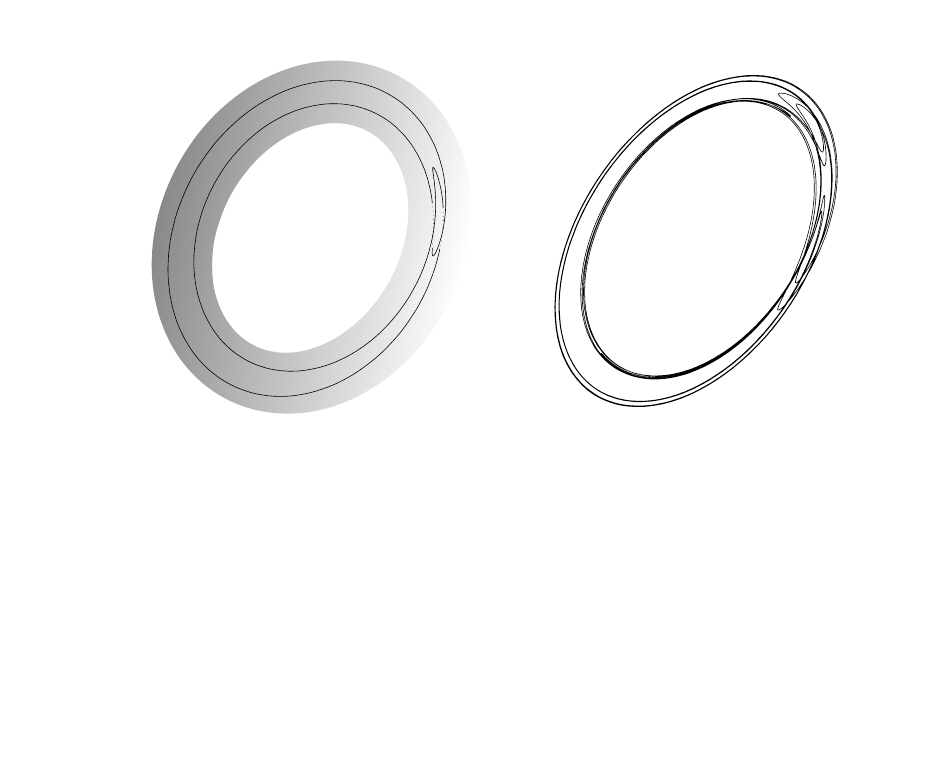}}
\psline[linewidth=0.1pt]{->}(5,0.6)(4.8,0.6)(4.2,1.4)
\rput[l](5.1,0.6){$\text{core of } T_2$}
\rput(11,1.6){$K$}
\rput(1,5.6){$T_1$}
\end{pspicture}
\caption{Whitehead's continuum}
\label{fig:whitehead}
\end{figure}

\end{example}

\subsection{The geometric degree} Let $K$ and $K'$ be two toroidal sets and $f$ a continuous map such that $f(K) = K'$ and $f$ is injective on some open neighbourhood $U$ of $K$. By invariance of domain this implies that $f$ is a homeomorphism onto the open set $U' = f(U)$.

Let $\{T_k\}$ and $\{T'_{\ell}\}$ be bases for $K$ and $K'$ respectively. Notice that for large enough $k$ the tori $T_k$ are contained in the domain of $f$ and in fact $\{f(T_k)\}$ is also a basis of $K'$: tameness of $f(T_k)$ is ensured by the discussion at the end of Subsection \ref{subsec:tameness} while the invariance of the geometric index under homeomorphisms ensures that $N(f(T_{k+1}) \subseteq f(T_k)) = N(T_{k+1} \subseteq T_k) \neq 0$.

Let $k$ be big enough so that $T_k$ is contained in the domain of $f$ and $f(T_k) \subseteq T'_1$. We define a rational number by \begin{equation} \label{eq:deg} d := \frac{N(f(T_k) \subseteq T'_1)}{N(T_k \subseteq T_1)} \in \mathbb{Q}. \end{equation} The denominator is nonzero because $N(T_k \subseteq T_1)$ is the product of the geometric indices of the consecutive pairs $T_2 \subseteq T_1$, $T_3 \subseteq T_2$, etc. and these are nonzero by definition of a basis.

\begin{proposition} $d$ does not depend on $k$.
\end{proposition}
\begin{proof} We provisionally denote $d$ by $d_k$ to reflect its dependence on $k$. Suppose that the condition $f(T_k) \subseteq T'_1$ holds. Then it also holds with $k+1$ instead of $k$, and to prove the proposition it suffices to show that $d_k = d_{k+1}$. We have \[d_{k+1} = \frac{N(f(T_{k+1}) \subseteq T'_1)}{N(T_{k+1} \subseteq T_1)}.\] Inserting $f (T_k)$ in the numerator and $T_k$ in the denominator using multiplicativity gives \[d_{k+1} = \frac{N(f(T_{k+1}) \subseteq T'_1)}{N(T_{k+1} \subseteq T_1)} = \frac{N(f(T_{k+1}) \subseteq f(T_{k})) N(f(T_{k}) \subseteq T'_1)}{N(T_{k+1} \subseteq T_{k}) N(T_{k} \subseteq T_1)}.\] By the invariance of the geometric index under homeomorphisms $N(f(T_{k+1}) \subseteq f(T_{k})) = N(T_{k+1} \subseteq T_{k})$ and so a cancellation occurs in the expression above, yielding $d_{k+1} = d_{k}$.
\end{proof}

We call $d$ the \emph{geometric degree} of $f$ with respect to the bases $\{T_k\}$ and $\{T'_{\ell}\}$ and denote it by $d_{\mathcal{N}}(f;\{T_k\},\{T'_{\ell}\})$. This degree is nonzero. This follows readily from Proposition \ref{prop:composition} below but can also seen directly. If $d_{\mathcal{N}}(f;\{T_k\},\{T'_{\ell}\}) = 0$ then $N(f(T_k) \subseteq T'_1) = 0$ for every $k$. For every $\ell$ we may choose $k$ large enough so that $f(T_k) \subseteq T'_{\ell}$, and then $0 = N(f(T_k) \subseteq T'_1) = N(f(T_k) \subseteq T'_{\ell}) N(T'_{\ell} \subseteq T'_1)$ implies $N(f(T_k) \subseteq T'_{\ell}) = 0$ since the second factor is nonzero by definition of a basis. But this implies that there is a ball $B_k$ between $f(T_k)$ and $T'_{\ell}$, and this being true for every $\ell$ would imply that $K'$ is cellular, a contradiction.

As with the usual degree, the geometric degree is multiplicative under composition. Suppose we have three toroidal sets $K$, $K'$, $K''$ with bases $\{T_k\}$, $\{T'_{\ell}\}$, $\{T''_m\}$ and local homeomorphisms $f,g$ that carry $K$ to $K'$ and $K'$ to $K''$ respectively.

\begin{proposition} \label{prop:composition} The geometric degree is multiplicative: \[d_{\mathcal{N}}(g \circ f;\{T_k\},\{T''_m\}) = d_{\mathcal{N}}(f;\{T_k\},\{T'_{\ell}\}) d_{\mathcal{N}}(g;\{T'_{\ell}\},\{T''_m\}).\]
\end{proposition}
\begin{proof} Choose $\ell$ so that $g(T'_{\ell}) \subseteq T''_1$ and then $k$ so that $f(T_{k}) \subseteq T'_{\ell}$. Then $(g \circ f)(T_{k}) \subseteq g (T'_{\ell})$, so again using the multiplicativity and invariance under homeomorphisms of the geometric index we have \begin{flalign*} N((g \circ f) (T_{k}) \subseteq T''_1) & = N((g \circ f)(T_{k} ) \subseteq g(T'_{\ell})) N(g(T'_{\ell}) \subseteq T''_1) \\ & = N(f(T_{k}) \subseteq T'_{\ell}) N(g(T'_{\ell}) \subseteq T''_1) \\ & = \frac{N(f(T_{k}) \subseteq T'_1)}{N(T'_{\ell} \subseteq T'_1)} N(g(T'_{\ell}) \subseteq T''_1) \\ & = N(f(T_{k})\subseteq T'_1) d_{\mathcal{N}}(g;\{T'_{\ell}\},\{T''_m\}).\end{flalign*} Dividing through by $N(T_{k} \subseteq T_1)$ gives \[\frac{N((g\circ f)(T_k)  \subseteq T''_1)}{N(T_k \subseteq T_1)} = d_{\mathcal{N}}(f;\{T_k\},\{T'_{\ell}\})  d_{\mathcal{N}}(g;\{T'_{\ell}\},\{T''_m\}).\] The left hand side is, by definition, $d_{\mathcal{N}}(g \circ f;\{T_k\},\{T''_m\})$, proving the desired equality.
\end{proof}

Now we particularize to the case when $K' = K$; i.e. $f$ is a local homeomorphism which leaves a toroidal set $K$ invariant. Then there is a natural choice of bases for the geometric degree of $f$; namely the same basis $\{T_k\}$ of $K$ for its role both as a source and target space.

\begin{proposition} \label{prop:d} Let $f$ be a local homeomorphism which leaves a toroidal set $K$ invariant. Then $d_{\mathcal{N}}(f;\{T_k\},\{T_k\})$ is independent of the basis $\{T_k\}$.
\end{proposition}
\begin{proof} Let $\{T_k\}$ and $\{T'_{\ell}\}$ be two bases for $K$. Observe that by multiplicativity of the geometric degree \[d_{\mathcal{N}}({\rm Id};\{T'_{\ell}\},\{T_k\})  d_{\mathcal{N}}({\rm Id};\{T_k\},\{T'_{\ell}\}) = d_{\mathcal{N}}({\rm Id};\{T'_{\ell}\},\{T'_{\ell}\}) = 1\] where the last equality is trivial from the definition. Writing $f = {\rm Id} \circ f \circ {\rm Id}$, again by multiplicativity \[d_{\mathcal{N}}(f;\{T'_{\ell}\},\{T'_{\ell}\}) = d_{\mathcal{N}}({\rm Id};\{T'_{\ell}\},\{T_k\}) d_{\mathcal{N}}(f;\{T_k\},\{T_k\}) d_{\mathcal{N}}({\rm Id};\{T_k\},\{T'_{\ell}\})\] and the first and third terms on the right hand side cancel each other by the previous paragraph.
\end{proof}

The previous proposition justifies the following definition:

\begin{definition} \label{defn:Ndeg} If $f$ is a local homeomorphism which leaves a toroidal set $K$ invariant and $\{T_k\}$ is any basis for $K$, we call $d_{\mathcal{N}}(f;\{T_k\},\{T_k\}) \in \mathbb{Q}$ the geometric degree of $f$ (on $K$) and denote it by $d_{\mathcal{N}}(f;K)$.
\end{definition}

\subsection{The case of attractors} We now discuss the case when $K$ is an attractor for $f$ (and not merely an invariant set as in Proposition \ref{prop:d}). The following theorem is the main result in this section. It relates a purely topological property of $K$ (its prime divisors) to the geometric degree of any attracting dynamics on it.

\begin{theorem} \label{thm:dattract} Let $K$ be a toroidal attractor for a local homeomorphism $f$. Then:
\begin{itemize}
    \item [(i)] $K$ has only finitely many prime divisors $p_i$ (possibly none).
    \item [(ii)] $d_{\mathcal{N}}(f;K)$ is an integer whose prime divisors in the ordinary sense of arithmetic are exactly those of $K$ (perhaps with multiplicities).
\end{itemize}

As a consequence $d_{\mathcal{N}}(f;K) \geq \prod_i p_i$, and $d_{\mathcal{N}}(f;K) = 1$ if and only if $K$ has no prime divisors.
\end{theorem}
\begin{proof} Let $T \subseteq \mathcal{A}(K)$ be a solid torus neighbourhood of $K$. Pick $r$ big enough so that $f^r(T) \subseteq T$ and set $n := N(f^r(T) \subseteq T)$. We are free to perform our computations in any convenient basis, and we choose the basis $\{T_k\}$ generated from $T$ by using the dynamics; namely $\{T \supseteq f^r(T) \supseteq f^{2r}(T) \supseteq \ldots\}$. Thus by definition $T_{k+1} = f^r (T_k)$, starting with $T_1 := T$.

(i) By construction each pair $T_{k+1} \subseteq T_k$ is homeomorphic to $f^r(T) \subseteq T$, and so $N(T_{k+1} \subseteq T_k) = N(f^r (T) \subseteq T) = n$ for every $k$. Thus a prime $p$ divides infinitely many of the $N(T_{k+1} \subseteq T_k)$ if and only if it divides $n$; in other words, the prime divisors of $K$ are precisely those of $n$. In particular, if $n = 1$ then $K$ has no prime divisors.

(ii) We first compute the geometric degree of $f^r$ instead of $f$. By definition  \[d_{\mathcal{N}}(f^r;K) = \frac{N(f^r(T_{k}) \subseteq T_1)}{N(T_{k} \subseteq T_1)} = \frac{N(T_{k+1} \subseteq T_1)}{N(T_k \subseteq T_1)}\] and, since $T_{k+1} \subseteq T_{k}$ by construction of the basis $\{T_k\}$, we may interpolate a $T_{k}$ in the numerator and cancel with the denominator to get $d_{\mathcal{N}}(f^r;K) = N(T_{k+1} \subseteq T_k) = n$.

By multiplicativity of the geometric degree we must have \[n = d_{\mathcal{N}}(f^r;K) = (d_{\mathcal{N}}(f;K))^r.\] An elementary argument shows that if an integer ($n$, which is an integer because it is a geometric index) has an $r$th root in $\mathbb{Q}$ (namely $d_{\mathcal{N}}(f;K)$) then in fact the root is an integer, so we deduce $d_{\mathcal{N}}(f;K) \in \mathbb{Z}$. Again using the equality $n = (d_{\mathcal{N}}(f;K))^r$ we see that the prime factors of $n$ must be exactly those of $d_{\mathcal{N}}(f;K)$.
\end{proof}

For later use we record here the following fact established in the proof of Theorem \ref{thm:dattract}:

\begin{remark} \label{rem:manning} If $T \subseteq \mathcal{A}(K)$ is any torus neighbourhood of $K$ and $r$ is such that $f^r(T) \subseteq T$, then $d_{\mathcal{N}}(f;K)^r = N(f^r(T) \subseteq T)$ and the prime divisors of $K$ are the prime divisors of $N(f^r(T) \subseteq T)$. When $T$ is positively invariant this is trivial from the definitions, but as mentioned in Remark \ref{rem:nottori} we do not know if a toroidal attractor always has a positively invariant neighbourhood which is a solid torus.
\end{remark}

\begin{example} \label{ex:whiteheadattractor} (1) A variation on the construction of Example \ref{ex:whiteheadtype} produces Whitehead continua which are automatically attractors. First fix a pair $(T_1,T_2)$ with the usual pattern. Now observe that there is a homeomorphism $f : \mathbb{R}^3 \longrightarrow \mathbb{R}^3$ that carries $T_1$ onto $T_2$. This must be the case since both are unknotted solid tori, but it is also easy to picture such an $f$. One starts with $T_1$, stretches it along some direction to obtain an tube shaped like an ellipse, then twists it one whole turn along its long axis, and finally folds it over. This makes the tube resemble the pattern shown in Figure \ref{fig:whitehead}. It is then just a matter of shrinking the diameter of the tube and moving it around to make it fit inside $T_1$. This whole procedure defines an isotopy of $\mathbb{R}^3$ whose final stage is the required homeomorphism $f$. Then setting $T_{k+1} := f^k(T_1)$ one obtains a Whitehead continuum $K := \bigcap_{k \geq 1} T_k$ which is by construction an attractor for $f$. We emphasize that the choice of $f$ is far from unique and can very well lead to different Whitehead continua or even to the same continuum but with essentially different dynamics. For example Garity, Jubran and Schori (\cite{gajuscho1}, \cite{gajuscho2}, \cite{jutesis}) show how to construct two homeomorphisms $f$ and $f'$ of $\mathbb{R}^3$ which both fit the framework just described and even have the same attractor $K$ but have the property that $f|_K$ is not transitive whereas $f'|_K$ is transitive (and in fact chaotic in the sense of Devaney).

(2) Whenever a local homeomorphism $f$ has a Whitehead continuum $K$ as an attractor, its geometric degree must be $d_{\mathcal{N}}(f;K) = 2,4,8, \ldots$ This is a consequence of Theorem \ref{thm:dattract} and the fact that $K$ has $2$ as its only prime divisor. All possible powers of $2$ can appear: for the specific construction of part (1) one has $d_{\mathcal{N}}(f;K) = 2$ (because $T_1$ is positively invariant and Remark \ref{rem:manning}.(1) applies with $r = 1$), and then the homeomorphisms $f^2, f^3, \ldots$ have the same continuum $K$ as an attractor and have geometric degrees $4,8,\ldots$
\end{example}

\subsection{Invariance under continuation} \label{subsec:continuation} Recall the setup for continuations: $U \subseteq \mathbb{R}^3$ is some open set and $f_{\lambda} : U \longrightarrow \mathbb{R}^3$ is a family of continuous, injective maps (hence homeomorphisms onto their image) which depends continuously on a parameter $\lambda \in [0,1]$.  We assume that $\lambda \longmapsto K_{\lambda}$ is a continuation of attractors for the $f_{\lambda}$. We begin by analyzing the local situation:

\begin{proposition} \label{prop:local} Let $\lambda \longmapsto K_{\lambda}$ be a continuation of attractors and suppose that $K_{\lambda_0}$ is toroidal. Then for $\lambda$ close enough to $\lambda_0$, the attractor $K_{\lambda}$ is also toroidal. Moreover, its prime divisors and the geometric degree of $f_{\lambda}$ at $K_{\lambda}$ are the same as those at $\lambda_0$.
\end{proposition}
\begin{proof} Let $N$ be a trapping region for $f_{\lambda_0}$. Since $K_{\lambda_0}$ is toroidal, there exists a (tame) solid torus $T \subseteq N$ which is a neighbourhood of $K_{\lambda_0}$. This $T$ is contained in the region of attraction of $K_{\lambda_0}$, so there exists an iterate $r \geq 1$ such that $f^r_{\lambda_0}(T) \subseteq {\rm int}\ T$. Finally, pick another trapping region $N'$ for $K_{\lambda_0}$ contained in $T$. There is an interval of parameters $I \subseteq [0,1]$ which is a neighbourhood of $\lambda_0$ and such that for every $\lambda \in I$ we still have the conditions (i) $f_{\lambda}(N) \subseteq {\rm int}\ N$, $f_{\lambda}(N') \subseteq {\rm int}\ N'$ (that is, $N$ and $N'$ are still trapping regions), (ii) the continuation $K_{\lambda}$ derived from $N$ and $N'$ is the same, and (iii) $f_{\lambda}^r(T) \subseteq {\rm int}\ T$.

Since $N \supseteq T \supseteq N'$, evidently \[\underbrace{\bigcap_{k \geq 0} f_{\lambda}^k(N)}_{=K_{\lambda}} \supseteq \bigcap_{k \geq 0} f_{\lambda}^k(T) \supseteq \underbrace{\bigcap_{k \geq 0} f_{\lambda}^k(N')}_{=K_{\lambda}}.\] and so $K_{\lambda} = \bigcap_{k \geq 0} f^k_{\lambda}(T)$. The same computation holds with $(f_{\lambda}^r)^k$ instead of $f_{\lambda}^k$. Thus we see that $K_{\lambda}$ is a toroidal set; in fact, it has the dynamical basis $\{(f_{\lambda}^r)^k(T)\}$. If we now show that $N(f_{\lambda}^r(T) \subseteq T)$ is independent of $\lambda$ the theorem will follow from Remark \ref{rem:manning}.(1).

Let $C := f_{\lambda_0}^r(T)$. Pick any $\lambda \in I$ and define an isotopy of $C$ inside $T$ by $h_t : C \longrightarrow T$ with $h_t := f^r_{\lambda_0+t(\lambda-\lambda_0)} \circ f_{\lambda_0}^{-r}$ and $t \in [0,1]$. This isotopy actually takes place in the interior of $T$, since $h_t(C) = f^r_{\lambda_0+t(\lambda-\lambda_0)}(T) \subseteq {\rm int}\ T$ because $\lambda_0 + t(\lambda-\lambda_0) \in I$ for every $t \in [0,1]$ and condition (iii) holds on $I$. Trivially $h_t$ can be extended to an open neighbourhood of $C$ in ${\rm int}\ T$; the same definition given above provides the extension. Now a deep result of Edwards and Kirby (\cite[Corollary 1.2, p. 63]{edwardskirby1}) ensures that the isotopy $h_t$ extends to an ambient isotopy, that is, there exists an isotopy $H_t : T \longrightarrow T$ such that $h_t = H_t|_C$. Then $H_1$ is a homeomorphism of $T$ which sends $C = f_{\lambda_0}^r(T)$ onto $h_1(C) = f_{\lambda}^r(T)$, and by the invariance of the geometric index under homeomorphisms we have $N(f^r_{\lambda_0}(T) \subseteq T) = N(f^r_{\lambda}(T) \subseteq T)$ as desired.
\end{proof}

The previous proposition shows in particular that the set $\{\lambda \in [0,1] : K_{\lambda} \text{ is toroidal}\}$ is open. In general it need not be closed, so to obtain a global continuation theorem one needs to place some extra assumption:

\begin{theorem} \label{thm:invariance} Let $\lambda \longmapsto K_{\lambda}$ be a continuation through toroidal sets (i.e. each $K_{\lambda}$ is toroidal). Then all the $K_{\lambda}$ have the same prime divisors and all the $f_{\lambda}$ have the same geometric degree at $K_{\lambda}$.
\end{theorem}
\begin{proof} This is completely straightforward. Since each $K_{\lambda}$ is toroidal by assumption, the previous proposition shows that the prime divisors of $K_{\lambda}$ and the degree of $f_{\lambda}$ at $K_{\lambda}$ depend on $\lambda$ in a locally constant manner. Since $[0,1]$ is connected, they must be constant.
\end{proof}

\section{Toroidal attractors with no prime divisors} \label{sec:noprimedivisors}

Suppose that $\gamma \subseteq \mathbb{R}^3$ is a smooth knot. We saw earlier (Example \ref{ex:basic}.(1)) that $\gamma$ is a toroidal set with no prime divisors, and it can be easily realized as an attractor for homeomorphism (even a $\mathcal{C}^{\infty}$ diffeomorphism) of $\mathbb{R}^3$. Thus smooth knots provide a particularly simple example of toroidal attractors with no prime divisors. In this section we show that, up to continuation, this is the only model for such attractors:

\begin{theorem} \label{thm:continuation} Let $K$ be a toroidal attractor for a homeomorphism $f$ of all $\mathbb{R}^3$. Then $K$ has no prime divisors if and only if it can be continued through toroidal attractors to a smooth knot $\gamma$. Moreover, the dynamics on $\gamma$ can be taken to be stationary.

If $f$ is $\mathcal{C}^{\infty}$ then the continuation can also be made $\mathcal{C}^{\infty}$.
\end{theorem}

\begin{remark} (1) Notice that the dynamical system $f$ is assumed to be a homeomorphism of all of $\mathbb{R}^3$ and not only a local homeomorphism as in the previous sections. Without this assumption it is 
easy to see that the theorem is false.

(2) The use of the geometric index (through the prime divisors of $K$) instead of the homological winding number is crucial. To see this consider again Example \ref{ex:comparedeg}. The toroidal attractor $K$ has no homological prime divisors but it cannot be continued to a smooth knot since it has $p=3$ as a geometric prime divisor.
\end{remark}

Implication $(\Leftarrow)$ of Theorem \ref{thm:continuation} is a direct consequence of Example \ref{ex:basic}.(1) and the invariance of prime divisors under continuation (Theorem \ref{thm:invariance}). The remaining of this section is devoted prove the converse implication. We will prove it in the $\mathcal{C}^{\infty}$ case. The (slightly simpler) argument in the topological case follows essentially the same steps and we will just make a couple of comments where appropriate.

Recall that a diffeotopy of $\mathbb{R}^3$ is a $\mathcal{C}^{\infty}$ map $G : [0,1] \times \mathbb{R}^3 \longrightarrow \mathbb{R}^3$ such that each partial map $G_t$ is a diffeomorphism of $\mathbb{R}^3$ and $G_0 = {\rm Id}$. A diffeotopy is supported on a set $U$ if each $G_t$ is the identity outside $U$. This implies in particular that $G_t(U) = U$ for each $t$. Diffeotopies can be concatenated; i.e. given two diffeotopies $G^{(1)}$ and $G^{(2)}$ one can first perform $G^{(1)}_{2t}$ for $0 \leq t \leq \nicefrac{1}{2}$ and then apply $G^{(2)}$ to the end result of the first diffeotopy, namely $G_{2t-1}^{(2)} \circ G_1^{(1)}$ for $\nicefrac{1}{2} \leq t \leq 1$. This is generally not differentiable at $t = \nicefrac{1}{2}$ but can be easily smoothed out (see \cite[p. 111]{hirsch1}).

We will use diffeotopies to produce $\mathcal{C}^{\infty}$ continuations as follows. Suppose that $g = g_0$ is a diffeomorphism of $\mathbb{R}^3$ which has an attractor $K$ with a trapping region $N$. If $G_t$ is a diffeotopy supported on $N$ then $[0,1] \ni \lambda \longmapsto G_{\lambda} \circ g$ produces a continuation of $g_0$, and the condition that $G$ be supported on $N$ implies that $g_{\lambda}(N) = G_{\lambda} \circ g(N) \subseteq G_{\lambda}({\rm int}\ N) = {\rm int}\ N$ for every $\lambda$. Thus $N$ is a trapping region throughout the whole continuation, and in particular its maximal invariant subset $K_{\lambda}$ is a continuation of $K$.

The proof of ($\Rightarrow$) of Theorem \ref{thm:continuation} requires two auxiliary lemmas. We state them, then prove the theorem, and finally prove the lemmas. Recall that two solid tori $T_1 \subseteq {\rm int}\ T_0$ are concentric if $\overline{T_0 \setminus T_1}$ is homeomorphic (or equivalently diffeomorphic, if the tori are smooth) to $\mathbb{T}^2 \times [0,1]$. The first auxiliary lemma states that given two concentric solid tori in $\mathbb{R}^3$ we can drag the smaller one via an ambient an ambient deformation until it matches the bigger one.

\begin{lemma} \label{lem:step1} Let $T_1 \subseteq {\rm int}\ T_0 \subseteq \mathbb{R}^3$ be two concentric smooth solid tori. There exists a diffeotopy $G$ of all $\mathbb{R}^3$ such that:
\begin{itemize}
    \item[(i)] $G_t(T_1) \subseteq T_0$ for all $t \in [0,1]$.
    \item[(ii)] $G_1(T_1) = T_0$.
\end{itemize}
Moreover, $G$ can be taken to be supported on any prescribed neighbourhood $U$ of $\overline{T_0 \setminus T_1}$.
\end{lemma}

The second lemma essentially states that if a diffeomorphism of $\mathbb{R}^3$ leaves a solid torus $T$ invariant and is homologically the identity on its boundary, then it can be deformed to be the identity on $T$.

\begin{lemma} \label{lem:step2} Let $T \subseteq \mathbb{R}^3$ be a differentiable solid torus. Let $g : \mathbb{R}^3 \longrightarrow \mathbb{R}^3$ be a diffeomorphism such that $g(T) = T$ and $g|_{\partial T}$ induces the identity in $H_1(\partial T;\mathbb{Z})$. Then there exists a diffeotopy $G$ of $\mathbb{R}^3$ such that:
\begin{itemize}
    \item[(i)] $G_t(T) = T$ for every $t \in [0,1]$.
    \item[(ii)] $G_1 \circ g|_T = {\rm Id}_T$.
\end{itemize}

Moreover, $G$ can be taken to be supported on any prescribed neighbourhood $U$ of $T$.
\end{lemma}

\begin{proof}[Proof of ($\Rightarrow$) in Theorem \ref{thm:continuation}] We assume that $K$ is a toroidal attractor for some diffeomorphism $f : \mathbb{R}^3 \longrightarrow \mathbb{R}^3$ and that $K$ has no prime divisors. We are going to show that some power of $f$ can be continued to a smooth curve with trivial dynamics.

Let $T$ be a smooth solid torus in $\mathcal{A}(K)$ which is a neighbourhood of $K$. Choose a power $r$ such that $f^r(T) \subseteq {\rm int}\ T$. By Remark \ref{rem:manning} the prime divisors of $K$ are the prime divisors of $N(f^r(T) \subseteq T)$, and so the latter must be $1$. Moreover, since $f$ is defined on all of $\mathbb{R}^3$, the tori $f^r(T)$ and $T$ are equivalently knotted. A result of Edwards (\cite[Theorem 3, p. 4]{chedwards1}) then implies that $f^r(T)$ and $T$ are concentric.

Set $g_0 := f^{2r}$, where the role of the extra factor $2$ in the exponent will be clear shortly. This diffeomorphism will be the starting point of our continuation of the dynamics. Notice that $g_0(T) \subseteq {\rm int}\ T$. Let $T'$ be a solid differentiable torus which is a thickening of $T$ thin enough so that $g_0(T') \subseteq {\rm int}\ T$ still.

Let $G$ be the diffeotopy given by Lemma \ref{lem:step1} applied to $T_0 = T$, $T_1 = g_0(T)$, and $U = T'$. Then $g_{\lambda} := G_{\lambda} \circ g_0$ defines a continuation of $g_0$ such that $g_1(T) = T$ and $T'$ is a trapping region for each $g_{\lambda}$ because $G$ is supported in $T'$.
\smallskip

{\it Claim.} The map $g_1|_{\partial T}$ induces the identity in $H_1(\partial T;\mathbb{Z})$.

{\it Proof of claim.} The homology group $H_1(\partial T;\mathbb{Z})$ is generated by the homology classes $m$ and $\ell$ of a meridian and a longitude of $\partial T$. The element $m$ is uniquely determined up to sign by the condition that when included in $T$ it becomes zero. The element $\ell$ depends on the framing of $T$, but one can fix it up to sign by requiring it to be nullhomologous in $\overline{\mathbb{R}^3 \setminus T}$  (this is called a preferred longitude in knot theory). Now, since $g_1$ is defined not only on $\partial T$ but in all of $\mathbb{R}^3$ and leaves both $T$ and $\overline{\mathbb{R}^3 \setminus T}$ invariant, $(g_1|_{\partial T})_*(m)$ and $(g_1|_{\partial T})_*(\ell)$ are again nullhomologous when included in $T$ and $\overline{\mathbb{R}^3 \setminus T}$ respectively, so $(g_1|_{\partial T})_*(m) = \pm m$ and $(g_1|_{\partial T})_*(\ell) = \pm \ell$ in $H_1(\partial T;\mathbb{Z})$.

Observe that both $g_0$ and $G_1$ are orientation preserving; the first because it is an even power of a homeomorphism and the second because it is isotopic to the identity. Thus $g_1$ is also orientation preserving, and so are $g_1|_T$ and consequently $g_1|_{\partial T}$. This implies that the determinant of the endomorphism $(g_1|_{\partial T})_*$ of $H_1(\partial T;\mathbb{Z})$ must be positive. The computation from the previous paragraph shows that in the basis $\{m,\ell\}$ the matrix of $(g_1|_{\partial T})_*$ is diagonal with $\pm 1$ entries; so we conclude that these must either both be positive or both be negative. Thus to prove that $(g_1|_{\partial T})_* = {\rm Id}$ we only need to show that the sign in $(g_1|_{\partial T})_*(\ell) = \pm \ell$ is actually a $+$. 

Consider the map $f^r|_T : T \longrightarrow T$. It can be regarded as the composition of $f^r$ and the inclusion $f^r(T) \subseteq T$, both of which induce isomorphisms in homology: the first because it is a homeomorphism; the second because $f^r(T) \subseteq T$ are concentric. Hence $(f^r|_T)_*$ is an isomorphism of $H_1(T;\mathbb{Z}) \cong \mathbb{Z}$, i.e. multiplication by $\pm 1$. Since by definition $g_0 = f^{2r}$, we have that $(g_0|_T)_*$ is the square of $(f^r|_T)_*$ so it follows that $(g_0|_T)_* = {\rm Id}$ in $H_1(T;\mathbb{Z})$. The map $G_t \circ g_0|_{T}$ is a homotopy between $g_0|_T$ and $g_1|_T$, and this homotopy takes place in $T$ by condition (i) of Lemma \ref{lem:step1} so $(g_1|_T)_* = (g_0|_T)_* = {\rm Id}$ in $H_1(T;\mathbb{Z})$.

We can now show that the sign in $(g_1|_{\partial T})_*(\ell) = \pm \ell$ is in fact a $+$. Regarding this as an equality in $H_1(T;\mathbb{Z})$ via the inclusion $\partial T \subseteq T$ we have $(g_1|_T)_*(\ell) = \pm \ell$, and since $(g_1|_T)_*  = {\rm Id}$ we get $\ell = \pm \ell$. Since $\ell \neq -\ell$ in $H_1(T;\mathbb{Z}) \cong \mathbb{Z}$ we conclude that the sign on the right hand side must be a $+$. $_{\blacksquare}$

\smallskip

Now we apply Lemma \ref{lem:step2} to $g_1$ and $U = T'$ to obtain a second diffeotopy $G_t$. Letting $\lambda$ run from $1$ to $2$ and setting $g_{\lambda} := G_{\lambda -1} \circ g_1$ we obtain a further continuation of $g_1$ to some $g_2$ such that $g_2|_T$ is the identity and again $T'$ is a trapping region for each $g_{\lambda}$.

Since $T'$ is a trapping region for $g_{\lambda}$ for every $\lambda \in [0,2]$ and its maximal invariant subset for $g_0$ is precisely $K$, the map $\lambda \longmapsto K_{\lambda} = \bigcap_{k \geq 0} g_{\lambda}^k(T')$ defines a continuation of $K$ through toroidal attractors. The attractor $K_2$ is contained in $T'$ and contains $T$ since the latter is compact and invariant under $g_2$. The last part of the proof consists in perturbing the dynamics further by gradually adding a ``radial'' component on $T'$ towards a core $\gamma$ of $T$. We first describe this idea in the abstract.

Consider the nested triple of solid tori $\mathbb{D}^2 \times \mathbb{S}^1 \subseteq (2 \mathbb{D}^2) \times \mathbb{S}^1 \subseteq (3 \mathbb{D}^2) \times \mathbb{S}^1$, where $r \mathbb{D}^2$ denotes the closed unit disk of radius $r$. Let $\rho$ be a diffeotopy of the interval $[0,3]$ which is supported on $[0,5/2]$ and such that $\rho_t|_{[0,2]}(r) = (1-t/2)r$. We use this to define a diffeotopy $f_t$ of the solid torus $(3 \mathbb{D}^2) \times \mathbb{S}^1$ by $f_t(x,s) := (\rho_t(\|x\|) x/\|x\|,s)$. The action of $f_t$ on any meridional disk $(2 \mathbb{D}^2) \times \{s\}$ is just given by $x \longmapsto (1-t/2)x$ so for $t \in (0,1]$ it is a radial contraction towards the origin. Hence $f_t$ sends $(2\mathbb{D}^2) \times \mathbb{S}^1$ into its interior and has the core $\{0\} \times \mathbb{S}^1$ as an attractor with stationary dynamics. The same is true of the restriction $f_t|_{\mathbb{D}^2 \times \mathbb{S}^1}$. At the final stage $t=1$ we have $f_1((2 \mathbb{D}^2) \times \mathbb{S}^1) = \mathbb{D}^2 \times \mathbb{S}^1$.

Now we copy this abstract construction to our setting. Recall that $T'$ was obtained by thickening $T$ very slightly. Let $T''$ be obtained from $T'$ in the same way and consider a diffeomorphism $h : (T'',T',T) \longrightarrow (3 \mathbb{D}^2, 2 \mathbb{D}^2,\mathbb{D}^2) \times \mathbb{S}^1$. Then we set $G$ to be the diffeotopy of $T''$ obtained by copying $f_t$ through $h$; i.e. $G_t(x) := h^{-1} \circ f_t \circ h(x)$. By construction each $G_t$ is the identity on a neighbourhood of $\partial T''$, and so we can extend $G$ to a diffeotopy of all $\mathbb{R}^3$. This has the following properties:
\begin{itemize}
    \item[(i)] $G_t(T') \subseteq T'$ for each $t \in [0,1]$ and $G_1(T') \subseteq T$.
    \item[(ii)] $G_1(T) \subseteq T$ and the restriction $G_1|_T$ is conjugate (via $h$) to a radial contraction of $\mathbb{D}^2 \times \mathbb{S}^1$ onto its core $\gamma := h^{-1}(\{0\} \times \mathbb{S}^1)$.
\end{itemize}

Letting $\lambda$ run from $2$ to $3$, consider the continuation $g_{\lambda} := G_{\lambda -2} \circ g_2$ of $g_2$. We have $g_{\lambda}(T') \subseteq {\rm int}\ T'$ by (i) and the corresponding property for $g_{2}$, so again the map $\lambda \longmapsto {\rm Inv}(T')$ is a continuation of $K_2$ through toroidal attractors to the attractor $K_{3}$ of $g_{3}$. We claim that $K_{3} = \gamma$. To check this first notice that for $x \in T$ we have $g_{3}(x) = G_1 \circ g_{2}(x) = G_1(x)$ which belongs to $G_1(T) \subseteq T$ again, so it follows inductively that $g_{3}^k|_T = G_1^k|_T = (G_1|_T)^k$ for every $k \geq 0$. Thus by (ii) the dynamics of $g_3$ on $T$ is conjugate (via $h$) to the dynamics of a radial contraction on $\mathbb{D}^2 \times \mathbb{S}^1$. The latter clearly has $\{0\} \times \mathbb{S}^1$ as an attractor, and so $\gamma$ is an attractor for $g_3|_T$. Since $T'$ is positively invariant under $g_{2}$ and $G_1(T') \subseteq T$, we have $g_{3}(T') \subseteq T$ and so $\gamma$ is an attractor also for $g_3|_{T'}$, so in particular $K_3 = \gamma$. Moreover, $g_{3}|_{\gamma} = G_1|_{\gamma} = {\rm Id}|_{\gamma}$.

The proof is complete. The full continuation from $K$ to $K_3$ is given by the (smoothed out) concatenation of the $G$s obtained in the succesive steps of the proof.
\end{proof}

Now we prove the auxiliary Lemmas \ref{lem:step1} and \ref{lem:step2}. The first is an easy exercise in differential topology and we will only sketch the proof:

\begin{proof}[Proof of Lemma \ref{lem:step1}] Let $C_1 \subseteq U$ be a closed collar of $\partial T_1$ inside $T_1$; $C$ the set $\overline{T_0 \setminus T_1}$, and $C_0 \subseteq U$ a closed collar of $\partial T_0$ in $\overline{\mathbb{R}^3 \setminus T_0}$. (In the topological category one needs to require that $T_0$ be tame to ensure that this last collar exists). Each of these sets is diffeomorphic to a $2$--torus $S$ times an interval: the first and third are collars; the second is a product because of the hypothesis that $T_0$ and $T_1$ are concentric. It is then a standard exercise to construct an ambient diffeotopy that is the identity outside $C_1 \cup C \cup C_0$ and stretches $C_1$ so much that it becomes $C_1 \cup C$, while shrinking $C \cup C_1$ appropriately so that they fit in $C_0$. (Pick a diffeomorphism $b : C_1 \cup C \cup C_0 \longrightarrow \mathbb{T}^2 \times [-1,2]$ such that $b(C_1) = \mathbb{T}^2 \times [-1,0]$, $b(C) = \mathbb{T}^2 \times [0,1]$ and $b(C_0) = \mathbb{T}^2 \times [1,2]$, so that $\partial T_1$ and $\partial T_0$ correspond to $\mathbb{T}^2 \times 0$ and $\mathbb{T}^2 \times 1$ respectively. Let $\rho : [-1,2] \longrightarrow [-1,2]$ be a diffeotopy which is stationary near $-1$ and $2$ and whose final stage sends $0$ to $1$ and $1$ to $\nicefrac{3}{2}$. Define a diffeotopy $a_t$ of $\mathbb{T}^2 \times [-1,2]$ by $a_t(x,s) := (x,\rho_t(s))$ and use this to construct a diffeotopy $G_t$ of $\mathbb{R}^3$  given by $G_t(x) = b^{-1} \circ a_t \circ b(x)$ if $x \in C_1 \cup C \cup C_0$ and the identity outside).
\end{proof}

For the second lemma we need the following result: any diffeomorphism $g : T \longrightarrow T$ of a solid torus $T$ such that $g|_{\partial T} = {\rm Id}_{\partial T}$ is diffeotopic to the identity ${\rm Id}_T$; i.e. there exists a diffeotopy $G$ of $T$ such that $G_1 \circ g  ={\rm Id}_T$. This is relatively easy to prove in the topological case, but in the differentiable category the proof is much more involved and in fact the result is equivalent to a conjecture of Smale which was settled by Hatcher in \cite{hatcher1} (for the formulation used here see \cite[(9), p. 606]{hatcher1}).

\begin{proof}[Proof of Lemma \ref{lem:step2}] Since diffeomorphisms of the $2$--torus are classified up to diffeotopy by the homomorphism they induce in homology (\cite[Theorem 2.5, p. 55]{farbmargalit}), it follows from the hypothesis that there exists a diffeotopy $G^{(1)}$ defined on $\partial T$ such that $G^{(1)}_1 \circ g|_{\partial T} = {\rm Id}_{\partial T}$. That diffeotopy can be extended, by using a collar of $\partial T$ in $T$, to a diffeotopy of all of $T$ which we still denote by $G^{(1)}$. Applying the result recalled before the proof to $G^{(1)}_1 \circ g$, there is a diffeotopy $G^{(2)}$ of $T$ such that $G^{(2)}_1 \circ G^{(1)}_1 \circ g = {\rm Id}_T$. Then the concatenation of $G^{(1)}$ and $G^{(2)}$ gives a diffeotopy $G$ of $T$ which carries $g|_T$ onto the identity. Using a sufficiently thin collar of $\partial T$ in $\overline{\mathbb{R}^3 \setminus T}$ one can extend the diffeotopy to all of $\mathbb{R}^3$ having it be the identity outside any prescribed neighbourhood of $T$. (The proof of the lemma in the topological case requires that one assumes $T$ to be tame so that $\partial T$ indeed has a collar in $\overline{\mathbb{R}^3 \setminus T}$).
\end{proof}

\section{The entropy of toroidal attractors} \label{sec:entropytoroidal}

The goal of this section is to prove that the geometric degree provides a lower bound on the entropy of a toroidal attractor. It is here where for the first time we need the dynamics to be smooth.

\begin{theorem} \label{thm:manning} Let $K$ be a toroidal attractor for a $\mathcal{C}^{\infty}$ local diffeomorphism $f$. Then $h(f|_K) \geq \log d_{\mathcal{N}}(f;K)$.
\end{theorem}

Coupling the above with the inequality $d_{\mathcal{N}}(f;K) \geq \prod_i p_i$ from Theorem \ref{thm:dattract} yields $h(f|_K) \geq \log d_{\mathcal{N}}(f;K) \geq \log \prod_i p_i$ which is bound \eqref{eq:intro1} from the Introduction. The rest of the argument leading to Theorem \ref{thm:intro} was already detailed there.

The proof of Theorem \ref{thm:manning} depends, in turn, on the following result concerning the entropy of dynamics on a solid torus. For definiteness we denote by $V \subset \mathbb{R}^3$ the solid torus obtained by rotating around the $z$--axis the disk of radius $\nicefrac{1}{2}$ and center $(0,\nicefrac{3}{2},0)$ contained in the $yz$--plane.

\begin{theorem} \label{thm:entropy} Let $f : V \longrightarrow V$ a $\mathcal{C}^{\infty}$ embedding. Then the entropy of $f$ is bounded below by $h(f) \geq \log N(f(V) \subseteq V)$.
\end{theorem}

\begin{proof}[Proof of Theorem \ref{thm:manning} from Theorem \ref{thm:entropy}] Let $T \subseteq \mathcal{A}(K)$ be a smooth solid torus neighbourhood of $K$ and $r$ big enough so that $f^r(T) \subseteq T$. Consider the restriction $f^r|_T : T \longrightarrow T$, which is a dynamical system on $T$ that still has $K$ as its (global in $T$) attractor. This implies that $K$ contains the non-wandering set of $f^r|_T$ and therefore by a result of Bowen (see \cite{bowen1} or \cite{robinson1}) it concentrates all the entropy: $h(f^r|_K) = h(f^r|_T)$. Now observe that Theorem \ref{thm:entropy} is valid not only for dynamics on $V$ but on any smooth solid torus. This is a direct consequence of the invariance of both entropy and the geometric index under conjugation. Thus for $f^r$ on the smooth solid torus $T$ we may write \[h(f^r|_T) \geq \log N(f^r(T) \subseteq T) =  r \log d_{\mathcal{N}}(f;K)\] where in the last step we have used Remark \ref{rem:manning}.(1). Finally, a standard property of entropy ensures that $h(f^r|_K) = r h(f|_K)$ and putting all this together yields $h(f|_K) \geq \log d_{\mathcal{N}}(f;K)$.
\end{proof}

It remains to prove Theorem \ref{thm:entropy}. The argument proceeds by showing that the length of a curve in a solid torus is bounded below by its geometric index (Lemma \ref{lemma:curvesbound}) and then applying an inequality of Yomdin (\cite{yomdin}) which relates volume growth rate and topological entropy for smooth dynamics. We only need a very particular case of the inequality, which we describe now.

Given a $\mathcal C^{\infty}$ path in $V$, $\sigma:[0,1]\longrightarrow V$, its length is given by the usual formula
\begin{equation*}
    \ell(\sigma):=\int_{[0,1]}\|\sigma'(t)\| dt.
\end{equation*}

Suppose $f : V \longrightarrow V$ is a $\mathcal{C}^{\infty}$ map (not necessarily injective). We consider the iterates of $\sigma$ under the dynamics generated by $f$; i.e. the paths $f^n \circ \sigma$, and the exponential growth rate of their lengths \[\overline{\lim}_{n\rightarrow \infty}\frac{1}{n}\log\ell(f^n\circ \sigma).\] 

The inequality that we need is the following:  \begin{equation} \label{eq:yomdin} \overline{\lim}_{n\rightarrow \infty}\frac{1}{n}\log\ell(f^n\circ \sigma) \leq h(f).\end{equation}

We have included a proof of \eqref{eq:yomdin} in Appendix \ref{app:yomdin}. This is done both for completeness and because, while in the general case the proof is very involved, in our particular setting it becomes fairly accessible while still retaining the essential ideas.

An extremely crude intuition of why \eqref{eq:yomdin} might be reasonable is as follows. Suppose $[0,1]$ is partitioned into intervals $I_i$ so that each portion $\sigma(I_i)$ of the curve is contained in an $(n,\epsilon)$-dynamical ball. This implies that the endpoints of $f^k \circ \sigma|_{I_i}$ lie at a distance less than $\epsilon$ for $k = 0,\ldots,n$ and, if the curve $f^k \circ \sigma|_{I_i}$ does not oscillate too much, its length will then be of order $\epsilon$. Thus the total length $\ell(f^k \circ \sigma)$ will be of the order of $\epsilon$ times the number of intervals $I_i$, which in turn is related to the minimal number of dynamical balls $S(n,\epsilon)$ needed to cover $V$ and therefore to the entropy $h(f)$. This heuristic idea breaks down if the curve $f^k \circ \sigma$ oscillates a lot. If it does so confined within a small ball the oscillations will have a large contribution to the length but not to the entropy. If it oscillates entering and exiting a dynamical ball a large number of times, then the number of intervals $I_i$ might grossly overestimate $S(n,\epsilon)$. The assumption on the smoothness of $f$ provides control over this phenomenon.

In order to apply \eqref{eq:yomdin} we need to bound from below the length of paths in the solid torus $V$. This is the content of the following lemma:

\begin{lemma}\label{lemma:curvesbound}
For every regular $\mathcal{C}^{\infty}$ parametrization $\gamma:[0,1] \longrightarrow V$ of a simple closed curve, its length $\ell(\gamma)$ is bounded as follows:
\begin{equation*}
    \ell(\gamma)\ge 2 \pi N(\gamma\subseteq V).
\end{equation*}
\end{lemma}

\begin{proof} We shall regard $\gamma$ as a diffeomorphic embedding $\gamma : \mathbb{S}^1 \longrightarrow V$ and denote its image also by $\gamma$.

Let $S \subseteq V$ be the shortest longitude of the solid torus $V$; namely the circumference in the $\{z=0\}$ plane, of radius $1$ and centered at the origin. Fix some orientation on $S$. We define a mapping $h : \mathbb{S}^1 \longrightarrow S$ which captures just the angular information in $\gamma$ as follows: $h$ is the composition of (i) the parameterization $\gamma$, followed by (ii) the orthogonal projection of $\mathbb R^3\setminus z\text{-axis}$ onto the punctured plane $\{z=0\}\setminus\{(0,0,0)\}$, followed finally by (iii) the radial retraction of the latter onto $S$. Each of these maps is differentiable so $h$ is differentiable as well.

Fix $0 < \epsilon < 2\pi$. By Sard's theorem, the set of critical values of $h$ can be covered by a family of open arcs $c_i$ whose lengths add up to less than $\epsilon$. The $c_i$ can be taken to be mutually disjoint and finite in number. It is possible that $h$ has no critical values: this happens precisely when $\gamma$ winds monotonically inside $V$, without doubling back. To avoid having to discuss that somewhat trivial case separately, we then take the family of arcs $\{c_i\}$ to consist of a single arc of length less than $\epsilon$.

Write $c_1, \ldots, c_n$ for the arcs in the covering with the convention that indices are taken cyclically (modulo $n$) as we move along $S$ in the positive orientation and denote by $e_i^1$ and $e_i^2$ the endpoints of $c_i$. Let $D_i^1$ and $D_i^2$ be the radial meridional disks of $V$ that go through the points $e_i^1$ and $e_i^2$ (that is, $D_i^j$ is the intersection of $V$ with the plane that contains the $z$-axis and the point $e_i^j$). Figure \ref{fig:cota} illustrates the definitions showing $V$ viewed from the top.

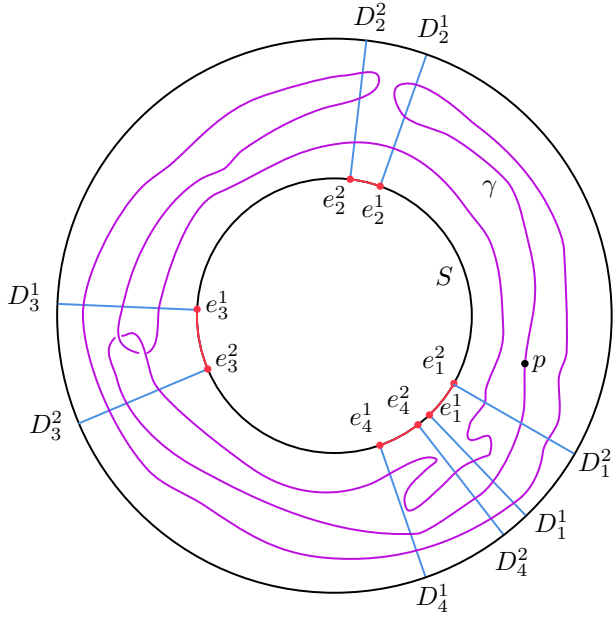
\begin{figure}[h]
    \centering
    \begin{tikzpicture}[x=0.75pt,y=0.75pt,yscale=-1,xscale=1]
\end{tikzpicture}
\tikzset{every picture/.style={line width=0.75pt}} 

\begin{tikzpicture}[x=0.75pt,y=0.75pt,yscale=-1,xscale=1]

\draw [color={rgb, 255:red, 74; green, 144; blue, 226 }  ,draw opacity=1 ]   (357,49.33) -- (346,80.51) -- (338.05,103.06) -- (333.6,115.67) ;
\draw [color={rgb, 255:red, 74; green, 144; blue, 226 }  ,draw opacity=1 ]   (333.45,246.58) -- (356.65,313.56) ;
\draw [color={rgb, 255:red, 74; green, 144; blue, 226 }  ,draw opacity=1 ]   (352.64,236.2) -- (395.85,291.76) ;
\draw   (170.67,181.17) .. controls (170.67,103.94) and (233.27,41.33) .. (310.5,41.33) .. controls (387.73,41.33) and (450.33,103.94) .. (450.33,181.17) .. controls (450.33,258.39) and (387.73,321) .. (310.5,321) .. controls (233.27,321) and (170.67,258.39) .. (170.67,181.17) -- cycle ;
\draw   (241.19,181.17) .. controls (241.19,142.89) and (272.22,111.86) .. (310.5,111.86) .. controls (348.78,111.86) and (379.81,142.89) .. (379.81,181.17) .. controls (379.81,219.44) and (348.78,250.47) .. (310.5,250.47) .. controls (272.22,250.47) and (241.19,219.44) .. (241.19,181.17) -- cycle ;
\draw [color={rgb, 255:red, 74; green, 144; blue, 226 }  ,draw opacity=1 ]   (326.65,42.16) -- (318.45,112.31) ;
\draw [color={rgb, 255:red, 74; green, 144; blue, 226 }  ,draw opacity=1 ]   (358.41,231.25) -- (407,282) ;
\draw [color={rgb, 255:red, 74; green, 144; blue, 226 }  ,draw opacity=1 ]   (171,175.2) -- (241.37,178.21) ;
\draw [color={rgb, 255:red, 74; green, 144; blue, 226 }  ,draw opacity=1 ]   (181.4,235.6) -- (246.78,208.03) ;
\draw  [color={rgb, 255:red, 189; green, 16; blue, 224 }  ,draw opacity=1 ][line width=0.75] [line join = round][line cap = round] (183.67,180.67) .. controls (183.67,161.25) and (200.17,129.5) .. (213.67,116) .. controls (218.48,111.19) and (224.68,107.96) .. (229.67,103.33) .. controls (243.2,90.77) and (254.05,78.18) .. (271.67,71.33) .. controls (280.8,67.78) and (306.48,60.33) .. (315,60) .. controls (321.01,59.77) and (332.25,55.36) .. (333,61.33) .. controls (334.29,71.68) and (312.2,74.04) .. (304.33,76.67) .. controls (289.92,81.47) and (276.03,89.17) .. (263.67,98) .. controls (260.69,100.13) and (257.3,105.85) .. (254.33,107.33) .. controls (250.42,109.29) and (246.02,110.3) .. (242.33,112.67) .. controls (223.3,124.9) and (210.09,149.94) .. (205,171.33) .. controls (201.57,185.75) and (196.76,196.27) .. (211.67,200) ;
\draw  [color={rgb, 255:red, 189; green, 16; blue, 224 }  ,draw opacity=1 ][line width=0.75] [line join = round][line cap = round] (216.13,200) .. controls (223.13,196.5) and (220.76,183.07) .. (222.13,176.67) .. controls (225.35,161.69) and (236.22,125.87) .. (251.5,119.33) .. controls (274.88,109.33) and (287.99,96.9) .. (315.57,94) .. controls (342.43,91.18) and (362.54,106.35) .. (376.31,126) .. controls (380.94,132.61) and (386.8,137.72) .. (389.65,145.33) .. controls (394.83,159.13) and (395.01,172.98) .. (395.66,188) .. controls (396.2,200.5) and (390.03,207.37) .. (386.98,218) .. controls (386.43,219.93) and (387.11,223.43) .. (386.32,225.33) .. controls (385.49,227.3) and (379.19,234.02) .. (378.98,235.33) .. controls (378.54,237.97) and (376.59,241.28) .. (378.31,243.33) .. controls (380.59,246.07) and (387.62,239.38) .. (388.99,242.67) .. controls (390.01,245.13) and (389.67,248.09) .. (388.99,250.67) .. controls (388.6,252.11) and (386.35,252.06) .. (384.98,252.67) .. controls (381.35,254.28) and (377.46,255.58) .. (374.3,258) .. controls (369.94,261.35) and (366.55,265.85) .. (362.29,269.33) .. controls (361.1,270.31) and (349.06,282.91) .. (346.94,278.67) .. controls (342.64,270.08) and (351.93,263.68) .. (356.95,258.67) .. controls (358.62,257) and (363.74,254.39) .. (361.62,253.33) .. controls (354.1,249.58) and (336.82,262.84) .. (328.92,266) .. controls (308.08,274.33) and (286.45,270.2) .. (269.52,258.67) .. controls (258.08,250.88) and (224.96,225.64) .. (219.46,214.67) .. controls (216.18,208.11) and (214.99,191.24) .. (208.78,190) .. controls (206.99,189.64) and (205.28,191.33) .. (203.44,191.33) ;
\draw  [color={rgb, 255:red, 189; green, 16; blue, 224 }  ,draw opacity=1 ][line width=0.75] [line join = round][line cap = round] (199.72,192.4) .. controls (193.16,199) and (198.98,205.57) .. (201.06,213.91) .. controls (205.49,231.72) and (223.05,247.92) .. (239.16,256.92) .. controls (273.25,275.96) and (311.31,293.36) .. (352.12,291.2) .. controls (357.63,290.9) and (377.4,275.87) .. (382.87,272.38) .. controls (393.97,265.27) and (405.68,239.45) .. (406.26,226.01) .. controls (406.81,213.36) and (405.81,212.6) .. (406.93,202.48) .. controls (409.21,181.9) and (415.39,172.64) .. (408.94,146.7) .. controls (406.45,136.71) and (404.32,123.28) .. (398.91,115.12) .. controls (394.71,108.79) and (373.83,92.26) .. (365.49,88.91) .. controls (360.01,86.7) and (349.64,84.39) .. (345.44,80.17) .. controls (342.66,77.38) and (337.61,67.22) .. (343.43,64.71) .. controls (355.51,59.51) and (378.49,82.45) .. (380.86,84.2) .. controls (394.7,94.38) and (406.65,106.33) .. (415.62,121.17) .. controls (419.16,127.03) and (425.91,139.85) .. (426.32,148.05) .. controls (427.16,165.06) and (428.21,182.1) .. (427.65,199.12) .. controls (427.49,204.01) and (423.81,208.35) .. (423.64,213.24) .. controls (423.42,219.95) and (424.95,226.81) .. (423.64,233.4) .. controls (422.43,239.49) and (416.39,244.04) .. (415.62,250.2) .. controls (414.45,259.66) and (409.99,268.72) .. (403.59,273.72) .. controls (375.88,295.4) and (342.42,306.84) .. (305.33,303.29) .. controls (279.93,300.86) and (265.25,290.54) .. (245.17,277.08) .. controls (235.81,270.81) and (224.09,267.23) .. (216.43,258.94) .. controls (199.32,240.41) and (183.68,204.18) .. (183.68,180.31) ;
\draw  [draw opacity=0] (246.61,208.06) .. controls (243.12,199.79) and (241.19,190.71) .. (241.19,181.17) .. controls (241.19,180.08) and (241.22,178.99) .. (241.27,177.91) -- (310.5,181.17) -- cycle ; \draw [color={rgb, 255:red, 255; green, 50; blue, 69}  ,draw opacity=1 ]   (246.61,208.06) .. controls (243.12,199.79) and (241.19,190.71) .. (241.19,181.17) .. controls (241.19,180.08) and (241.22,178.99) .. (241.27,177.91) ; \draw [shift={(241.27,177.91)}, rotate = 267.74] [color={rgb, 255:red, 255; green, 50; blue, 69}  ,draw opacity=1 ][fill={rgb, 255:red, 255; green, 50; blue, 69}  ,fill opacity=1 ][line width=0.75]      (0, 0) circle [x radius= 1.34, y radius= 1.34]   ; \draw [shift={(246.61,208.06)}, rotate = 252.21] [color={rgb, 255:red, 255; green, 50; blue, 69}  ,draw opacity=1 ][fill={rgb, 255:red, 255; green, 50; blue, 69}  ,fill opacity=1 ][line width=0.75]      (0, 0) circle [x radius= 1.34, y radius= 1.34]   ;
\draw  [draw opacity=0] (318.45,112.31) .. controls (323.69,112.91) and (328.75,114.09) .. (333.56,115.79) -- (310.5,181.17) -- cycle ; \draw [color={rgb, 255:red, 255; green, 50; blue, 69}  ,draw opacity=1 ]   (318.45,112.31) .. controls (323.69,112.91) and (328.75,114.09) .. (333.56,115.79) ; \draw [shift={(333.56,115.79)}, rotate = 14.49] [color={rgb, 255:red, 255; green, 50; blue, 69}  ,draw opacity=1 ][fill={rgb, 255:red, 255; green, 50; blue, 69}  ,fill opacity=1 ][line width=0.75]      (0, 0) circle [x radius= 1.34, y radius= 1.34]   ; \draw [shift={(318.45,112.31)}, rotate = 11.52] [color={rgb, 255:red, 255; green, 50; blue, 69}  ,draw opacity=1 ][fill={rgb, 255:red, 255; green, 50; blue, 69}  ,fill opacity=1 ][line width=0.75]      (0, 0) circle [x radius= 1.34, y radius= 1.34]   ;
\draw  [draw opacity=0] (370.79,215.37) .. controls (367.45,221.26) and (363.26,226.6) .. (358.41,231.25) -- (310.5,181.17) -- cycle ; \draw [color={rgb, 255:red, 255; green, 50; blue, 69}  ,draw opacity=1 ]   (370.79,215.37) .. controls (367.45,221.26) and (363.26,226.6) .. (358.41,231.25) ; \draw [shift={(358.41,231.25)}, rotate = 131.29] [color={rgb, 255:red, 255; green, 50; blue, 69}  ,draw opacity=1 ][fill={rgb, 255:red, 255; green, 50; blue, 69}  ,fill opacity=1 ][line width=0.75]      (0, 0) circle [x radius= 1.34, y radius= 1.34]   ; \draw [shift={(370.79,215.37)}, rotate = 124.64] [color={rgb, 255:red, 255; green, 50; blue, 69}  ,draw opacity=1 ][fill={rgb, 255:red, 255; green, 50; blue, 69}  ,fill opacity=1 ][line width=0.75]      (0, 0) circle [x radius= 1.34, y radius= 1.34]   ;
\draw  [draw opacity=0] (352.64,236.2) .. controls (346.88,240.61) and (340.42,244.14) .. (333.45,246.58) -- (310.5,181.17) -- cycle ; \draw [color={rgb, 255:red, 255; green, 50; blue, 69}  ,draw opacity=1 ]   (352.64,236.2) .. controls (346.88,240.61) and (340.42,244.14) .. (333.45,246.58) ; \draw [shift={(333.45,246.58)}, rotate = 155.53] [color={rgb, 255:red, 255; green, 50; blue, 69}  ,draw opacity=1 ][fill={rgb, 255:red, 255; green, 50; blue, 69}  ,fill opacity=1 ][line width=0.75]      (0, 0) circle [x radius= 1.34, y radius= 1.34]   ; \draw [shift={(352.64,236.2)}, rotate = 147.55] [color={rgb, 255:red, 255; green, 50; blue, 69}  ,draw opacity=1 ][fill={rgb, 255:red, 255; green, 50; blue, 69}  ,fill opacity=1 ][line width=0.75]      (0, 0) circle [x radius= 1.34, y radius= 1.34]   ;
\draw [color={rgb, 255:red, 74; green, 144; blue, 226 }  ,draw opacity=1 ]   (370.35,215.64) -- (404.59,235.36) -- (431.4,250.8) ;
\draw  [fill={rgb, 255:red, 0; green, 0; blue, 0 }  ,fill opacity=1 ] (405.21,205.36) .. controls (405.21,204.6) and (405.83,203.98) .. (406.59,203.98) .. controls (407.35,203.98) and (407.96,204.6) .. (407.96,205.36) .. controls (407.96,206.12) and (407.35,206.73) .. (406.59,206.73) .. controls (405.83,206.73) and (405.21,206.12) .. (405.21,205.36) -- cycle ;

\draw (353.3,195) node [anchor=north west][inner sep=0.75pt]    {$e_{1}^{2}$};
\draw (362,219) node [anchor=north west][inner sep=0.75pt]    {$e_{1}^{1}$};
\draw (316.57,224) node [anchor=north west][inner sep=0.75pt]    {$e_{4}^{1}$};
\draw (244,167) node [anchor=north west][inner sep=0.75pt]    {$e_{3}^{1}$};
\draw (335.47,215) node [anchor=north west][inner sep=0.75pt]    {$e_{4}^{2}$};
\draw (248.9,194) node [anchor=north west][inner sep=0.75pt]    {$e_{3}^{2}$};
\draw (350.5,28) node [anchor=north west][inner sep=0.75pt]    {$D_{2}^{1}$};
\draw (318.5,22) node [anchor=north west][inner sep=0.75pt]    {$D_{2}^{2}$};
\draw (144,163) node [anchor=north west][inner sep=0.75pt]    {$D_{3}^{1}$};
\draw (155,228) node [anchor=north west][inner sep=0.75pt]    {$D_{3}^{2}$};
\draw (410,278) node [anchor=north west][inner sep=0.75pt]    {$D_{1}^{1}$};
\draw (432,247) node [anchor=north west][inner sep=0.75pt]    {$D_{1}^{2}$};
\draw (383.5,111.5) node [anchor=north west][inner sep=0.75pt]    {$\gamma $};
\draw (322.7,119) node [anchor=north west][inner sep=0.75pt]    {$e_{2}^{1}$};
\draw (303,114) node [anchor=north west][inner sep=0.75pt]    {$e_{2}^{2}$};
\draw (409,200) node [anchor=north west][inner sep=0.75pt]    {$p$};
\draw (390,297) node [anchor=north west][inner sep=0.75pt]    {$D_{4}^{2}$};
\draw (350,315) node [anchor=north west][inner sep=0.75pt]    {$D_{4}^{1}$};

\draw (360,155) node [anchor=north west][inner sep=0.75pt]    {$S$};

\end{tikzpicture}
    \caption{Setup for Lemma \ref{lemma:curvesbound} \label{fig:cota}}
    
\end{figure}

Fix any one of the indices $i$ and denote by $V_i$ the closed sector of $V$ comprised between the disks $D_i^2$ and $D_{i+1}^1$. Suppose $t \in \mathbb{S}^1$ is such that $p := \gamma(t)$ belongs to the interior of the sector $V_i$. Then:
\smallskip

{\it Claim.} $p$ lies in an arc of $\gamma$ which is entirely contained in the sector $V_i$ and joins the disks $D_i^2$ and $D_{i+1}^1$.

{\it Proof of claim.} Define $J \subseteq \mathbb{S}^1$ to be the closure of the connected component of $\mathbb{S}^1 \backslash \gamma^{-1}(\bigcup_{l,m} D_l^m)$ which contains $t$. Notice that $\gamma^{-1}(\bigcup_{l,m} D_l^m)$ is closed in $\mathbb{S}^1$, so it is disjoint from ${\rm int}\ J$ and contains $\partial J$. Hence $\gamma(J)$ is an arc in $V$ whose interior is disjoint from $\bigcup_{l,m} D_l^m$ and whose endpoints are contained in $\bigcup_{l,m} D_l^m$. Since the arc is connected and contains $p$, it must be completely contained in the sector $V_i$ and its endpoints must lie in $D_i^2 \cup D_{i+1}^1$. Thus $h(J)$ is an arc in $S$ whose endpoints are contained in $\{e_i^2,e_{i+1}^1\}$ and which is disjoint from all the $c_j$. The latter implies that $h|_J$ has no critical points, and so in particular it cannot map both endpoints of $J$ onto the same point $e_i^2$ or $e_{i+1}^1$ (otherwise it would reach some local extremum on $J$). Hence the endpoints of $\gamma(J)$ indeed lie one in $D_i^2$ and the other in $D_{i+1}^1$. ${\blacksquare}$
\smallskip

Now pick a point $p_0 \in D_i^2 \cap \gamma$. The following holds:
\smallskip

{\it Claim.} $p_0$ is an endpoint of an arc of $\gamma$ which is entirely contained in the sector $V_i$ and whose other endpoint is contained in $D_{i+1}^1$.

{\it Proof of claim.} Write $p_0 = \gamma(t_0)$ for some $t_0 \in \mathbb{S}^1$. Since $\gamma$ is transverse to the disk $D_i^2$, there exists a small arc $I \subseteq \mathbb{S}^1$ centered at $t_0$ such that $\gamma(I)$ intersects $D_i^2$ precisely at $p_0$ and both components of $\gamma(I \backslash \{t_0\})$ lie on different sides of $D_i^2$. Since the latter is the common boundary of $V_i$ and $V_{i-1}$, one of the components $I_0$ of $I \backslash \{t_0\}$ satisfies that $\gamma(I_0)$ is contained in the interior of $V_i$. It only remains to pick any $t \in I_0$ and apply the previous claim to $p := \gamma(t)$. ${\blacksquare}$
\smallskip

The two claims above imply that the number of arcs of $\gamma$ in the sector $V_i$ is precisely $|D_i^2 \cap \gamma|$, and each of them has a length at least $ d_i^{i+1}$ where $d_i^{i+1}$ is the angle between $D_i^2$ and $D_{i+1}^1$. Since $|D_i^2 \cap \gamma| \geq N(\gamma \subseteq V)$ by the definition of the geometric index, the total contribution of these arcs to the length of $\gamma$ is bounded below by $d_i^{i+1}N(\gamma\subseteq V)$. Hence

\begin{equation*}
    \ell(\gamma) \ge N(\gamma\subseteq V)\sum_{i}d_i^{i+1} \geq N(\gamma \subseteq V) (2 \pi - \varepsilon)
\end{equation*} where the last inequality makes use of the fact that $\sum_{i}d_i^{i+1}\ge 2\pi-\varepsilon$ by the choice of the covering $\{c_i\}$. Since this is true for every $\varepsilon > 0$ the lemma follows.
\end{proof}

The proof of Theorem \ref{thm:entropy} is now straightforward. Recall that $f : V \longrightarrow V$ is a $\mathcal{C}^{\infty}$ embedding. Let $\sigma$ be a smooth core of $V$ (for example its centerline).  By the preceding lemma applied to $\gamma = f^n \circ \sigma$ \[\ell(f^n \circ \sigma) \geq 2\pi N(f^n \circ \sigma \subseteq V) = 2 \pi N(f^n(V) \subseteq V)\] where the last equality owes to the definition of the geometric index of a curve and the fact that $f^n \circ \sigma$ is a core of $f^n(V)$. Observe that for every $n \geq 1$ we have $f^n(V) \subseteq f^{n-1}(V) \subseteq \ldots \subseteq f(V) \subseteq V$ and so by multiplicativity and invariance under homeomorphisms of the geometric index $N(f^n(V) \subseteq V) = \prod_{k=0}^{n-1} N(f^{k+1}(V) \subseteq f^k(V)) = \prod_{k=0}^{n-1} N(f(V) \subseteq V) = N(f(V) \subseteq V)^n$. Therefore $\ell(f^n \circ \sigma) \geq 2\pi N(f(V) \subseteq V)^n$ and now by inequality \eqref{eq:yomdin} \begin{equation*} h(f) \geq 
   \overline{\lim}_{n\rightarrow \infty}\frac{1}{n}\log \ell(f^n\circ\sigma)\ge\log N(f(V) \subseteq V).
\end{equation*} as was to be shown.

\section{Concluding remarks} \label{sec:concluding}

\subsection{} The formal similarities between the homological winding number $m$ and the geometric index $N$ imply that one can obtain ``homological'' counterparts to the definitions and results above by replacing $N$ with $m$ throughout. Thus one can define homological prime divisors $q_i$ of a toroidal set $K$, a homological degree $d(f;K)$, etc. These turn out to be describable in terms of \v{C}ech cohomology. One needs to assume that $\check{H}^1(K;\mathbb{Z}) \neq 0$ for the definitions to make sense. Then:

\begin{itemize}
    \item[(D1)] The homological prime divisors of $K$ correspond to those prime numbers $q$ which divide every element in $\check{H}^1(K;\mathbb{Z})$ in the sense that for every $x \in \check{H}^1(K;\mathbb{Z})$ there exists $y \in \check{H}^1(K;\mathbb{Z})$ such that $x = y + \stackrel{(q)}{\ldots} + y$.
    \item[(D2)] Suppose $f$ is a local homeomorphism which leaves a toroidal set $K$ invariant. Then the induced homomorphism $f^* : \check{H}^1(K;\mathbb{Z}) \longrightarrow \check{H}^1(K;\mathbb{Z})$ turns out to be multiplication by the homological degree $d(f;K)$.
\end{itemize}

A straightforward adaptation of the proof of Theorem \ref{thm:dattract} yields:

\begin{itemize}
    \item[(D3a)] When $K$ is an attractor for $f$, the homological degree $d(f;K)$ is an integer whose prime divisors are exactly the prime divisors $q_i$ of $K$. Moreover,
    \item[(D3b)] if $T \subseteq \mathcal{A}(K)$ is a solid torus neighbourhood of $K$ and $r$ is such that $f^r(T) \subseteq T$ then the $q_i$ are precisely the prime divisors of $|m(f^r(T) \subseteq T)|$ and the homological degree $d(f;K)$ satisfies $|d(f;K)|^r = |m(f^r(T) \subseteq T)|$.
\end{itemize}

Since $|m| \leq N$ in general, (D3b) and Remark \ref{rem:manning} yield the inequality $|d(f;K)| \leq d_{\mathcal{N}}(f;K)$.

For a $\mathcal{C}^0$ embedding $f : V \longrightarrow V$ the induced map $f_* : H_1(V) \longrightarrow H_1(V)$ is just multiplication by $m(f(V) \subseteq V)$, and so a classical theorem of Manning (\cite{manning1}) implies that $h(f) \geq |m(f(V) \subseteq V)|$. This is the homological counterpart to Theorem \ref{thm:entropy}, where notably the smoothness assumption is not needed. Replacing $N$ with $m$ in the derivation of Theorem \ref{thm:manning} from Theorem \ref{thm:entropy} shows that:

\begin{itemize}
    \item[(D4)] If $K$ is a toroidal attractor for a local homeomorphism and $\check{H}^1(K;\mathbb{Z}) \neq 0$ then $h(f|_K) \geq \log |d(f;K)|$.
\end{itemize}

Combining (D3a) and (D4) one has $d(f;K) \geq \prod_i q_i$ and as a consequence the universal bound \begin{equation} \label{eq:boundc0} h(f|_K) \geq \log |d(f;K)| \geq \log \prod_i q_i\end{equation} where $f$ is now a local homeomorphism having a toroidal set $K$ as an attractor. This is structurally similar to \eqref{eq:intro1} and holds without any smoothness assumptions.

 \begin{example} Let $K \subseteq \mathbb{R}^3$ be the standard embedding of an $n$--adic solenoid in $\mathbb{R}^3$ (see for example \cite[\S 17.1]{katokhasselblatt}). This particular embedding is toroidal and its prime divisors, both geometric and homological, are just the prime divisors $p_i = q_i$ of the integer $n$ (counted just once). Thus if $f$ is any local homeomorphism having $K$ as an attractor then $h(f|_K) \geq \log \prod_i q_i$. There exists a homeomorphism of $\mathbb{R}^3$ (in fact, a $\mathcal{C}^{\infty}$ diffeomorphism) which realizes the embedded solenoid $K$ as an attractor with its standard dynamics. These have an entropy $\log \prod_i q_i$, and so in this case the universal bound $\log \prod_i q_i$ for any attracting dynamics on $K$ is in fact sharp.
 \end{example}

In spite of the previous example, the homological approach has several shortcomings. It is not applicable when $\check{H}^1(K;\mathbb{Z}) = 0$, so for instance it leaves out examples as simple as the Whitehead continua described in Example \ref{ex:whiteheadtype}. Also, Theorem \ref{thm:continuation} fails if one replaces (geometric) prime divisors with homological prime divisors. The consequence of this is that the bound $h(f|_K) \geq \log \prod_i q_i$ is not sharp even up to continuation, and so the homological degree is not powerful enough to prove Theorem \ref{thm:intro}. The following example provides an illustration of these phenomena.

\begin{example} \label{ex:comparedeg} Consider a variation of Example \ref{ex:whiteheadattractor}.(1) where $T_1$ and $T_2 = f(T_1)$ are still unknotted and the core of $f(T_1)$ looks like the pattern in the left panel of Figure \ref{fig:whitehead} but with an extra full winding inside $T_1$. The resulting toroidal set $K$ can still be realized as an attractor for a $\mathcal{C}^{\infty}$ diffeomorphism $f$ of $\mathbb{R}^3$. Now $m(T_2 \subseteq T_1) = 1$ and $N(T_2 \subseteq T_1) = 3$. Therefore:
\begin{itemize}
    \item $K$ has no homological prime divisors and $d(f;K) = 1$.
    \item $K$ has $p=3$ as a (geometric) prime divisor, and $d_{\mathcal{N}}(f;K) = 3$.
\end{itemize}

This illustrates that the inequality $|d(f;K)| \leq d_{\mathcal{N}}(f;K)$ can be strict. Since $f$ is $\mathcal{C}^{\infty}$, according to Theorem \ref{thm:manning} we have $h(f|_K) \geq \log d_{\mathcal{N}}(f;K) = \log 3$ whereas the homological bound only says $h(f|_K) \geq \log d(f;K) = \log 1 = 0$. Notice that $K$ cannot be continued to a smooth knot through toroidal attractors (because it has $3$ as a prime divisor) but the homological bound on the entropy vanishes.
\end{example}

\subsection{} It is natural to ask whether the bound $h(f|_K) \geq \log d_{\mathcal{N}}(f;K)$ is valid when $f$ is just a (local) homeomorphism, without any smoothness assumptions. The derivation of Theorem \ref{thm:manning} from Theorem \ref{thm:entropy} works equally well in the $\mathcal{C}^0$ case, so this boils down to the following:
\smallskip

{\it Question. Let $f : V \longrightarrow V$ be continuous and injective. Is it true that $h(f) \geq \log N(f(V) \subseteq V)$?}
\smallskip

A heuristic argument which might suggest an affirmative answer goes as follows. Let again $\sigma$ be the centerline of $V$. After some technical work it can be arranged that the angular behaviour of $f^n \circ \sigma$ be piecewise monotone, and then the number of monotonicity intervals is bounded below by $N(f^n \circ \sigma \subseteq V) = N(f(V) \subseteq V)^n$. A classical result of Misiurewicz and Szlenk (\cite[Theorem 1, p. 48]{misiurewiczszlenk}) concerning the entropy of self-maps of $\mathbb{S}^1$ then suggests that this angular behaviour alone should contribute $\log N(f(V) \subseteq V)$ to the entropy of $f$.

If the answer to the question posed above is affirmative then, since $d_{\mathcal{N}}(f;K) \geq \prod_i p_i$ holds generally, one would have the universal bound $h(f|_K) \geq \log \prod_i p_i$ for any local homeomorphism $f$ having $K$ as an attractor. In turn this leads to the following theorem which improves Theorem \ref{thm:intro} from the Introduction in that it gives a conclusion about $K$ itself, with no continuations involved, and makes no smoothness assumptions:
\smallskip

{\it Theorem. Let $K$ be a toroidal attractor for a homeomorphism $f$ of $\mathbb{R}^3$. Then either $K$ admits stationary attracting dynamics or every attracting dynamics on $K$ has an entropy at least $\log 2$.}
\smallskip

\begin{proof} If $K$ has a prime divisor then $h(f|_K) \geq \log 2$. If $K$ has no prime divisors, as in the proof of Theorem \ref{thm:continuation} this implies that $K$ has a basis of concentric solid tori. Using these one can construct a $\mathcal{C}^0$ flow which has $K$ as an attractor (see \cite[Theorem B, p. 71]{hecyo2}) and is stationary on it. In particular, the time-one map of such a flow provides a homeomorphism of $\mathbb{R}^3$ which has $K$ as an attractor and is stationary on it.
\end{proof}

\appendix

\section{} \label{app:index}

In this Appendix we relate our topological definition of the geometric index to the standard piecewise linear (pl) one and show how the properties (P2) and (P3) from p. \pageref{pg:properties} can be derived from their pl counterparts. Since we will be working simultaneously with the topological and the piecewise linear geometric indices, we shall distinguish them notationally with a subscript thus: $N_{\rm top}$ and $N_{\rm pl}$.

We recall the pl definition. Let $T \subseteq \mathbb{R}^3$ be a polyhedral solid torus and $\gamma \subseteq {\rm int}\ T$ a polygonal simple closed curve. We consider all possible polyhedral meridional disks $D$ of $T$ (which may be very contorted) and count the number of points of intersection in $D \cap \gamma$. The geometric index of $\gamma$ inside $T$ is defined as \[N_{\rm pl}(\gamma \subseteq T) := \min \{|D \cap \gamma| : D \text{ is a polyhedral meridional disk of } T\}.\]

It should be intuitively clear that this number is finite: since the objects involved are polyhedral, a slight perturbation of any meridional disk $D$ will always make it transverse to $\gamma$ and then $D \cap \gamma$ will consist of finitely many points.

Now suppose $T_1 \subseteq T_0$ is a pair of polyhedral solid tori with $T_1$ contained in the interior of $T_0$. Let $\gamma_1$ be a polyhedral core of $T_1$. Then one defines the geometric index of $T_1$ inside $T_0$ as $N_{\rm pl}(T_1 \subseteq T_0) := N_{\rm pl}(\gamma_1 \subseteq T_0)$. This definition is correct because cores are unique up to isotopy. An equivalent definition which does not make use of cores and is more directly related to our definition of $N_{\rm top}$ is the following. Consider a meridional disk $D$ of $T_0$. By perturbing it slightly if necessary we may achieve that it intersects $\partial T_1$ along disjoint simple closed curves $\gamma_i$ which bound meridional disks $E_i$ of $T_1$. Then $N_{\rm pl}(T_1 \subseteq T_0)$ is the smallest possible number of these curves. The equivalence of this definition and the previous one is a consequence of \cite[Hilfssatz 5, p. 174]{schubert}. 

For a polyhedral pair of tori $T_1 \subseteq T_0$ we now have two definitions of a geometric index. The following proposition shows that both are equivalent.

\begin{proposition} Suppose $T_1 \subseteq T_0$ is a pair of polyhedral solid tori. Then $N_{\rm top}(T_1 \subseteq T_0) = N_{\rm pl}(T_1 \subseteq T_0)$.
\end{proposition}
\begin{proof} Let $p$ and $t$ be the polyhedral and topological geometric indices of $T_1$ inside $T_0$, respectively. As mentioned above, there is a pl meridional disk of $T_0$ which intersects $T_1$ in precisely $p$ meridional disks. This can clearly be made transverse to $T_1$ by slightly perturbing its vertices if necessary; hence $t \leq p$.

To prove the converse we shall show that any topological meridional disk $D$ of $T_0$ which is transverse to $T_1$ in the sense of Definition \ref{defn:Ntop} can be used to construct a polyhedral disk which is still transverse to $T_1$ and intersects it in the same number of meridional disks, thus showing that $p \leq t$.

Since $T_1$ is polyhedral, the set $M := T_0 \backslash {\rm int}\ T_1$ is a compact $3$--manifold whose boundary is the disjoint union of $\partial T_0$ and $\partial T_1$. Consider $D^* := D \cap M$, which is a disk with holes. The boundaries of these holes are simple closed curves $\gamma_1, \ldots, \gamma_k$ in $\partial T_1$ which are meridional curves of $T_1$.

By the topological transversality condition in Definition \ref{defn:Ntop} the set $D^*$ is semi-locally tame and hence tame (\cite[Theorem 9, p. 157]{bing2}). Thus there exists a homeomorphism $h$ of $M$ onto itself which moves points less than any prescribed $\epsilon$ and sends $D^*$ to a polyhedral disk with holes. Due to the invariance of the boundary of a manifold under homeomorphisms, $h$ sends $\partial T_0$ and $\partial T_1$ to themselves. By choosing $\epsilon$ sufficiently small, each $h(\gamma_i)$ is homotopic to $\gamma_i$ in $\partial T_1$; in particular each is still a meridional curve (non-nullhomotopic in $\partial T_1$ but nullhomotopic in $T_1$). Thus each $h(\gamma_i)$ bounds a polyhedral meridional disk $E_i$ of $T_1$; these can clearly be taken to be disjoint by removing their possible intersections.  Then $h(D^*) \cup E_1 \cup \ldots \cup E_k$ is a polyhedral meridional disk of $T_0$ which intersects $T_1$ in $k$ disks.
\end{proof}

The analogues of properties (P1) to (P3) in p. \pageref{pg:properties} were established, in the pl context, by Schubert. The above proposition then implies that they also hold in the topological context. We shall argue this for the multiplicativity property, for example.

Consider three nested solid tori $T_2 \subseteq T_1 \subseteq T_0$, where as usual we take the tori to be tame. Pick any homeomorphism $h_0$ from $T_0$ onto any polyhedral solid torus $P_0$. This sends $h_0(T_1)$ onto a topological solid torus inside $T_0$ which is semilocally tame (it being the image under a local homeomorphism of a tame object). Thus it is tame; in fact there is a homeomorphism $h_1 : P_0 \longrightarrow P_0$ which is the identity outside a neighbourhood of $h_0(T_1)$ and sends the latter onto a polyhedral solid torus $P_1$. Thus the homeomorphism $h_1 \circ h_0$ sends both $T_0$ and $T_1$ onto polyhedral tori. Applying the same argument once more we obtain a homeomorphism $h$ which sends the triple $(T_0,T_1,T_2)$ onto a triple of polyhedral solid tori $(P_0,P_1,P_2)$. By the invariance of $N_{\rm top}$ under homeomorphisms and the proposition above $N_{\rm top}(T_j \subseteq T_i) = N_{\rm top}(P_j \subseteq P_i) = N_{\rm pl}(P_j \subseteq P_i)$. Then the pl multiplicativity property $N_{\rm pl}(P_2 \subseteq P_0) = N_{\rm pl}(P_2 \subseteq P_1) \cdot N_{\rm pl}(P_1 \subseteq P_0)$ implies directly its topological counterpart $N_{\rm top}(T_2 \subseteq T_0) = N_{\rm top}(T_2 \subseteq T_1) \cdot N_{\rm top}(T_1 \subseteq T_0)$.

\section{Yomdin's inequality} \label{app:yomdin}

In this Appendix we prove Yomdin's inequality \eqref{eq:yomdin} for our very specific case of a map $f : V \longrightarrow V$ and exponential growth rates of lengths of curves.  We have followed the papers by Yomdin \cite{yomdin} and Gromov \cite{gromov}. Although we will eventually take $f$ to be $\mathcal{C}^{\infty}$, for the proof we work with a fixed degree of smoothness $k$. Define $M(f)$ to be the maximum of $\|df\|_{\infty}$ over $V$, and $R(f)$ as the limit
\begin{equation*}
    R(f):=\lim_{n\rightarrow\infty}\frac{1}{n}\log \max_{x\in V}\|d f^n\|_{\infty},
\end{equation*}
or in other words the exponential growth rate of $M(f^n)$. Notice that $M(f^{n+m}) \leq M(f^n) M(f^m)$ so the sequence $\log M(f^n)$ is subadditive. Thus $R(f)$ exists and satisfies $R(f) \leq M(f) < +\infty$.

For the sake of brevity we shall denote by ${\rm EGR}\ a_n$ the exponential growth rate of any sequence $a_n \geq 0$; i.e. $\overline{\lim}_{n\rightarrow \infty}\frac{1}{n}\log a_n$. Now, the topological entropy of $f$ is bounded as follows:
\begin{theorem}\label{thm:yomdin}
    Let $f : V \longrightarrow V$ be of class $\mathcal{C}^k$. For any $\mathcal{C}^k$ curve $\sigma : [0,1] \longrightarrow V$ the following inequality holds:
    \begin{equation} \label{eq:yomdin0}
        {\rm EGR}\ \ell(f^n \circ \sigma) \leq h(f) + \frac{1}{k} R(f).
    \end{equation}
\end{theorem}

Notice that $R(f)$ is finite and does not depend on $k$. Hence when $f$ and $\sigma$ are of class $\mathcal{C}^{\infty}$ then $k$ can be taken to be arbitrarily large and so the $R(f)/k$ term can be removed, leading to inequality \eqref{eq:yomdin}.

\begin{remark} Inequality \eqref{eq:yomdin0} actually holds for arbitrary compact manifolds and simplices of arbitrary dimension instead of paths, replacing length with the appropriate volume. This is the celebrated theorem of Yomdin. Combining this with an inequality of Newhouse (\cite{newhouse}) shows that when $f$ is $\mathcal{C}^{\infty}$ its entropy $h(f)$ is in fact equal to the supremum of the left hand side of \eqref{eq:yomdin} over all smooth simplices of arbitrary dimensions.
\end{remark}

Examination of \eqref{eq:yomdin0} under time rescaling (i.e. replacing $f$ with $f^q$ for some fixed $q \geq 1$) leads to the observation that it is enough to prove the following a priori slightly weaker bound: \begin{equation} \label{eq:yomdin_weak} {\rm EGR}\ \ell(f^n \circ \sigma) \leq h(f) + A(k) + \frac{1}{k}M(f)\end{equation} where $A(k)$ is some number which only depends on $k$ but not on $f$ or $\sigma$. This implies inequality \eqref{eq:yomdin0} as follows. Fix the path $\sigma : [0,1] \longrightarrow V$ and pick some sequence $n_k \rightarrow +\infty$ such that $1/n_k \log \ell(f^{n_k} \circ \sigma)$ converges to the exponential growth rate ${\rm EGR}\ \ell(f^n \circ \sigma)$. Fix some $q \geq 1$ and write $n_k = m_kq + r_k$ with $0 \leq r_k < q$. After passing to a subsequence we may take all $r_k$ to be equal to some fixed $r$. Trivially \[\frac{1}{n_k} \log \ell(f^{n_k} \circ \sigma) = \frac{m_k}{n_k} \frac{1}{m_k} \log \ell((f^q)^{m_k} \circ (f^r \circ \sigma))\] and taking limits as $k \rightarrow +\infty$ yields \[{\rm EGR}\ \ell(f^n \circ \sigma) \leq \frac{1}{q} {\rm EGR}\ \ell((f^q)^m \circ (f^r \circ \sigma)),\] where the exponential growth rate on the right hand side is taken as $m \rightarrow +\infty$. Now the weaker inequality \eqref{eq:yomdin_weak} applied to the map $f^q$ and the path $f^r \circ \sigma$ on the right hand side yields \[{\rm EGR}\ \ell(f^n \circ \sigma) \leq \frac{1}{q} \left( h(f^q) + A(k) + \frac{1}{k} \log M(f^q) \right).\] A standard property of entropy ensures that $h(f^q) = qh(f)$. Plugging this above and letting $q \rightarrow +\infty$ removes the constant term $A(k)$ and leads to \eqref{eq:yomdin} by the definition of $R(f)$.

\subsection{Basic strategy of the proof} Ultimately, the entropy is related to the minimum number of dynamical balls needed to cover $V$. One possible way of estimating this from below is the following.

Fix some curve $\sigma : [0,1] \longrightarrow V$ and any open covering $\mathcal{B}$. For each member $B$ of the covering we want to analyze the intersection of (the image of) $\sigma$ with $B$, which will generally be a union of curves, and how the length of these grows as we iterate them with $f$. Thus we consider the set $\sigma^{-1}(B) \subseteq [0,1]$ which is a disjoint union of compact intervals (its connected components), and the length of $f^n \circ \sigma|_{\sigma^{-1}(B)}$ for each $B \in \mathcal{B}$. Obviously as $B$ ranges over the covering $\mathcal{B}$ the sets $\sigma^{-1}(B)$ provide a covering of $[0,1]$, and so $\ell(f^n \circ \sigma) \leq \sum_{B \in \mathcal{B}} \ell(f^n \circ \sigma|_{\sigma^{-1}(B)})$. In particular when $\mathcal{B}$ is a subcover of $\mathcal{B}(n,\epsilon)$ with the minimal number of elements, namely $S(n,\epsilon)$, we have \[\ell(f^n \circ \sigma) \leq S(n,\epsilon) \sup_{B \in \mathcal{B}(n,\epsilon)} \ell(f^n \circ \sigma|_{\sigma^{-1}(B)})\] because the number of summands is $S(n,\epsilon)$ and every summand can be overestimated by the supremum of the lengths over the whole cover $\mathcal{B}(n,\epsilon)$ instead of only the minimal subcover. Taking exponential growth rates on both sides and then letting $\epsilon \rightarrow 0$ leads to \begin{equation} \label{eq:app_basic1} {\rm EGR}\ \ell(f^n \circ \sigma) \leq h(f) + \lim_{\epsilon \rightarrow 0} {\rm EGR}\ \sup_{B \in \mathcal{B}(n,\epsilon)} \ell(f^n \circ \sigma|_{\sigma^{-1}(B)}).\end{equation} Comparing this with \eqref{eq:yomdin_weak} we see that one needs to show that the second summand is bounded above by an expression of the form $A(k) + M(f)/k$. The key to doing this is an analysis of the set $\sigma^{-1}(B)$. This is provided by Lemma \ref{lem:fund1} below. We introduce a preliminary definition and then state the lemma.

If $\tau : [0,1] \longrightarrow V$ is a curve such that all the derivatives of $\tau$ up to order $k$ have a norm $\leq 1$ we shall say that $\tau$ is normalized. If $\tau$ is defined only on a subinterval $J \subseteq [0,1]$ we reparameterize it by letting $\psi : [0,1] \longrightarrow J$ be the unique affine bijection which preserves orientation, and say that $\tau$ is normalized if $\tau \circ \psi$ is normalized. Clearly the length of a normalized curve is bounded above by $1$.

\begin{lemma} \label{lem:fund1} Let $f : V \longrightarrow V$ be a $\mathcal{C}^k$ map. There exist numbers $\mu(k)$ and $\epsilon_0(f,k)$ with the following property. For any $\mathcal{C}^k$ curve $\sigma : [0,1] \longrightarrow V$ and any $(n,\epsilon)$-dynamical ball $B \subseteq V$ with $\epsilon < \epsilon_0$ the set $\sigma^{-1}(B)$ can be covered by at most $c(\sigma,\epsilon) (\mu(k) M(f)^{\nicefrac{1}{k}})^n$ intervals $J_i$ such that for each of these the curve $f^n \circ \sigma|_{J_i}$ is normalized.
\end{lemma}

The notation $c(\sigma,\epsilon)$ means that $c$ depends only on $\sigma$ and $\epsilon$ but not on $f$ or $B$. It follows immediately from the lemma that for $\epsilon < \epsilon_0(f,k)$ \[\ell(f^n \circ \sigma|_{\sigma^{-1}(B)}) \leq \sum_i \ell(f^n \circ \sigma|_{J_i}) \leq c(\sigma,\epsilon) (\mu(k) M(f)^{\nicefrac{1}{k}})^n\] since each piece $f^n \circ \sigma|_{J_i}$ is normalized. The right hand side of the above inequality is independent of $B$, and so the supremum of these lengths over $B \in \mathcal{B}(n,\epsilon)$ has the same upper bound. Therefore extracting the exponential growth rate one has \[ {\rm EGR}\ \sup_{B \in \mathcal{B}(n,\epsilon)} \ell(f^n \circ \sigma|_{\sigma^{-1}(B)}) \leq \log \mu(k) + \frac{1}{k} \log M(f) \] whenever $\epsilon < \epsilon_0(f,k)$. In particular the inequality holds in the limit $\epsilon \rightarrow 0$, so coupled with \eqref{eq:app_basic1} we get \[ {\rm EGR}\ \ell(f^n \circ \sigma) \leq h(f) + \log \mu(k) + \frac{1}{k}\log M(f) \] which is valid for any $\sigma : [0,1] \longrightarrow V$. With $A(k) = \log \mu(k)$ this has the desired form \eqref{eq:yomdin_weak} and therefore proves \eqref{eq:yomdin0} via time rescaling as explained above.

\subsection{The inductive step for Lemma \ref{lem:fund1}} Suppose $t \in [0,1]$ belongs to $\sigma^{-1}(B)$ where $B$ is an $(n,\epsilon)$-dynamical ball. By definition, there is some $x$ such that $\sigma(t) \in B(x,\epsilon)$, $f \circ \sigma(t) \in B(f(x),\epsilon)$, etc. up to $f^n \circ \sigma(t) \in B(f^n(x),\epsilon)$. Thus finding the set $\sigma^{-1}(B)$ can be approached iteratively as follows. For each $r$ let $S_r \subseteq [0,1]$ be the set of parameters $t$ such that $\sigma(t) \in B(x,\epsilon), f \circ \sigma(t) \in B(f(x),\epsilon)\ldots$, up to iterate $r$. Clearly the $S_r$ form a decreasing sequence of compact sets and $\sigma^{-1}(B) = S_n$. Suppose we have already found some $S_r$. To find $S_{r+1}$ one may proceed as follows. Each $S_r$ is a disjoint union of closed intervals (its connected components; maybe some degenerate). If we write $S_r = \bigcup J$ for that decomposition, one can describe $S_{r+1}$ by restricting attention to each $(f^r \circ \sigma)|_J$ at a time and finding what parameter values $t \in J$ satisfy the additional condition $f^{r+1} \circ \sigma(t) \in B(f^{r+1}(x),\epsilon)$. In other words, letting $\tau$ be the curve $f^r \circ \sigma|_J$, which is contained in the ball $B(z,\epsilon)$ with $z = f^r(x)$, we need to find the set of parameters $t \in J$ such that $f \circ \tau(t) \in B(f(z),\epsilon)$. This set might be quite complicated (for instance, it might have infinitely many connected components) and it will be technically convenient to overestimate it slightly; that is, we will actually show that it is contained in a union of closed intervals which we shall be able to bound in number.

The proof of Lemma \ref{lem:fund1} thus boils down to a one-step estimate. Much as we did with the time rescaling earlier on, by taking advantage of a degree of freedom related to the metric size of the problem we can reduce the one-step estimate to the convenient case when the derivatives of $f$ of orders $2, \ldots, k$ are bounded above by $1$ and the radius of the ball $B$ is $\epsilon = 1$. We therefore state the one-step estimate for a solid torus $V_{\lambda}$ which is just the solid torus $V$ scaled up by a factor $\lambda \geq 1$:

\begin{lemma} \label{lem:noepsilon} Let $f_{\lambda} : V_{\lambda} \longrightarrow V_{\lambda}$ be a  $\mathcal{C}^k$ map with $\|d^s f\|_{\infty} \leq 1$ for every $s = 2, \ldots, k$. Let $\sigma_{\lambda} : [0,1] \longrightarrow V_{\lambda}$ be a $\mathcal{C}^k$ normalized curve and $B \subseteq V_{\lambda}$ a closed ball of radius $1$ and an arbitrary center. Finally, let $S = \{t \in [0,1] : f_{\lambda} \circ \sigma_{\lambda}(t) \in B\}$. Then there exists a family of no more than $\mu(k) M(f_{\lambda})^{\nicefrac{1}{k}}$ intervals $J_i$ such that:
\begin{itemize}
    \item[(i)] The set $S$ is covered by the $J_i$.
    \item[(ii)] For each of the $J_i$ the curve $f_{\lambda} \circ \sigma_{\lambda}|_{J_i}$ is normalized.
\end{itemize}
\end{lemma}

Applying this one-step case to each of the curves $f_{\lambda} \circ \sigma_{\lambda}|_{J_i}$ (after the standard affine reparametrization) and repeating this inductively as described earlier we obtain that for an $(n,1)$-dynamical ball the set $\sigma^{-1}(B)$ is covered by at most $(\mu(k) M(f_{\lambda})^{\nicefrac{1}{k}})^n$  intervals $J_{i_1 i_2 \ldots i_n}$ such that for each of them the curve $f^n_{\lambda} \circ \sigma_{\lambda}|_{J_{i_1 i_2 \ldots i_n}}$ is normalized.

\begin{proof}[Proof of Lemma \ref{lem:fund1} from Lemma \ref{lem:noepsilon}] First we scale up all the elements in our problem by a factor $\lambda \geq 1$ to be fixed later. The solid torus $V$ becomes the solid torus $V_{\lambda} := {\lambda} \cdot V$ and the map $f$ is replaced with its conjugate via the rescaling; i.e. $f_{\lambda} : V_{\lambda} \longrightarrow V_{\lambda}$ given by $f_{\lambda}(x) := {\lambda} \cdot f(x/{\lambda})$. Notice that this rescaling does not change the first derivatives of $f$ (so $M(f) = M(f_{\lambda})$) but divides its $s$th derivatives, $s \geq 1$, by a factor ${\lambda}^{s-1}$. Hence by choosing ${\lambda}$ large enough we can achieve that all the derivatives of $f_{\lambda}$ of orders $2,3,\ldots,k$ be bounded above by $1$ on $V_{\lambda}$. The cutoff value $\lambda_0$ of $\lambda$ at which this happens depends only on the derivatives of $f$ on $V$. Let $\epsilon_0(f,k) := 1/{\lambda_0}$.

Now consider an $(n,\epsilon)$-dynamical ball $B \subseteq V$, with $\epsilon < \epsilon_0$. It is straightforward to see that its scaled up version $\lambda \cdot B$ is just an $(n,\lambda \epsilon)$-dynamical ball for $f_{\lambda}$. Choose $\lambda := \nicefrac{1}{\epsilon}$ as the scaling factor, so that $\lambda B$ is actually an $(n,1)$-dynamical ball. Since $\epsilon < \epsilon_0$ we have $\lambda > \lambda_0$ and so the derivatives of $f_{\lambda}$ of order $2 \leq s \leq k$ are all bounded by $1$. This sets us under the assumptions of Lemma \ref{lem:noepsilon} save for the normalization of the curve $\sigma$. This is the origin of the $c(\sigma,\epsilon)$ factor in the statement of the lemma, as follows. Split $[0,1]$ into $N$ intervals $I_j$ of equal length $1/N$. It is then trivial to check that the standard reparameterization $\sigma_{\lambda} \circ \psi_j$ satisfies $\|d^s \sigma_{\lambda} \circ \psi_j\|_{\infty} = \lambda \|d^s \sigma\|_{\infty}/N^s$. Set $c(\sigma,\epsilon) := \max_{1 \leq s \leq k} (\lambda \|d^s \sigma\|_{\infty})^{\nicefrac{1}{s}} = \max_{1 \leq s \leq k} (\epsilon^{-1} \|d^s \sigma\|_{\infty})^{\nicefrac{1}{s}}$. Clearly by choosing $N \approx c(\sigma,\epsilon)$, by which we mean $c(\sigma,\epsilon)$ rounded up to the closest integer, we have that $\sigma_{\lambda} \circ \psi_j$ is normalized for all $j$. Then an inductive application of Lemma \ref{lem:noepsilon} as described before to each piece $\sigma_{\lambda} \circ \psi_j$ in turn shows that each $S_j := \{t \in [0,1] : \sigma_{\lambda} \circ \psi_j(t) \in \lambda \cdot B\}$ can be covered by at most $(\mu(k) M(f_{\lambda})^{\nicefrac{1}{k}})^n= (\mu(k) M(f)^{\nicefrac{1}{k}})^n$ intervals $J_{j i_1 i_2 \ldots i_n}$. Scaling back down to $V$ one has $S_j = \{t \in [0,1] : \sigma \circ \psi_j(t) \in B\}$, and so $\sigma^{-1}(B) = \bigcup_j S_j \subseteq \bigcup_{j i_1 \ldots i_n} J_{j i_1 \ldots i_n}$ is then covered by at most $c(\sigma,\epsilon) (\mu(k) M(f)^{\nicefrac{1}{k}})^n$ intervals.
\end{proof}

\subsection{The one-step estimate} Here we finally prove Lemma \ref{lem:noepsilon}. The basic strategy of the proof consists in replacing $f \circ \sigma$ with local Taylor approximations $Q$ of degree $k$ and solving the problem for these.
\medskip

{\it The polynomial case.} In all the following statements $Q : [0,1] \longrightarrow \mathbb{R}^3$ is a polynomial curve of degree $k$ which will eventually be a Taylor approximation to $f \circ \sigma$. Any number expressed in the form $\alpha(k)$ is meant to depend only on $k$ and not on $Q$, $\lambda$, etc.

(1) Let $B' \subseteq V_{\lambda}$ be a ball of radius $\nicefrac{3}{2}$ and set \[S' = \{ t \in [0,1] : Q(t) \in B'\}.\] Then $S'$ can be covered by at most $\alpha_1(k)$ intervals $J_i$ with the property that $Q(J_i)$ is contained in a ball of radius $2$.

{\it Proof.} $B'$ is the intersection of an Euclidean ball with the solid torus $V_{\lambda}$. These can be described by polynomial inequations of degrees $2$ and $4$ respectively, and so $S'$ is described by two polynomial inequalities of degrees $2k$ and $4k$. The set $S'$ is therefore a union of intervals (possibly degenerated) whose endpoints are roots of those polynomials or $0$ or $1$. Thus $S'$ consists of at most $\alpha_1(k) := 6k+2$ connected components (if all of them are singletons; if all of them are nondegenerate intervals there can be at most $3k+1$ of these). If a component of $S'$ is a singleton, we just enlarge it ever so slightly that its image under $Q$ is contained in the ball of radius $2$ concentric with $B'$.
\medskip

(2) Assume that the image of $Q$ is contained in a ball of radius $2$. Then $[0,1]$ can be partitioned into $\alpha_2(k)$ intervals $J_i$ such that the reparameterized curves $Q \circ \psi_i$ satisfy the $\nicefrac{1}{2}$--normalization condition $\|d^s (Q \circ \psi_i)\|_{\infty} \leq \nicefrac{1}{2}$ for $s = 1,\ldots,k$.

{\it Proof.} A result of Markov (see for example \cite{schduff}) implies that there exists a constant $c(k)$ such that for every polynomial $Q : [0,1] \longrightarrow \mathbb{R}^3$ of degree $k$ one has $\|d^s Q\|_{\infty} \leq c(k) \|Q\|_{\infty}$ for $s = 1,\ldots,k$. Under our assumptions this is bounded above by $2c(k)$. Thus splitting $[0,1]$ into $N \approx 4 c(k)$ intervals $J_i$ of equal length and reparameterizing affinely via $\psi_i$, each curve $Q_i := Q \circ \psi_i$ will satisfy $\|d^s(Q \circ \psi_i)\|_{\infty} = N^{-s} \|d^s Q\|_{\infty} \leq N^{-1} \|d^s Q\|_{\infty} \leq N^{-1} 2c(k) \leq \nicefrac{1}{2}$ for $s = 1, \ldots, k$.
\medskip

Combining (1) and (2) we get the following:

\begin{lemma} \label{lem:poly} Let $Q : [0,1] \longrightarrow \mathbb{R}^3$ be a polynomial curve of degree $k$. Let $B'$ be a ball of radius $\nicefrac{3}{2}$ in $V_{\lambda}$. Then the set \[S' = \{ t \in [0,1] : Q(t) \in B'\}\] can be covered by at most $\alpha(k)$ intervals $J_i$ such that $Q|_{J_i}$ is $\nicefrac{1}{2}$--normalized for each $i$.
\end{lemma}
\begin{proof} By (1) the set $S'$ can be covered by $\alpha_1(k)$ intervals $J_i$. For each of these we consider the reparameterized curve $Q \circ \psi_i$. Its image is contained in a ball of radius $2$ as stated in (1), and so applying (2) to $Q \circ \psi_i$ the interval $[0,1]$ can be split into no more than $\alpha_2(k)$ intervals such that $(Q \circ \psi_i)$ restricted to each of them is $\nicefrac{1}{2}$-normalized.
\end{proof}

\medskip

{\it Taylor approximations.} The basic strategy of the proof of Lemma \ref{lem:noepsilon} will consist in partitioning the interval $[0,1]$ into subintervals and replacing $f \circ \sigma$ with its Taylor polynomial of degree $k$ on each of those subintervals. We now make some comments about these approximations. Recall that we have a $\mathcal{C}^k$ map $f_{\lambda} : V_{\lambda} \longrightarrow V_{\lambda}$ with $\|d^s f_{\lambda}\|_{\infty} \leq 1$ for $s = 2, \ldots, k$ and also a $\mathcal{C}^k$ normalized curve $\sigma_{\lambda} : [0,1] \longrightarrow V_{\lambda}$. For notational ease we will henceforth drop the subindex $\lambda$ from all objects except for $V_{\lambda}$ itself.

(1) The derivatives of order $s$ of $f \circ \sigma$ can be expressed via the chain rule and, using the fact that $f$ has derivatives $\leq 1$ for orders $s \geq 2$ and $\sigma$ has derivatives $\leq 1$ for all orders, we see that for every $s = 1, \ldots, k$ one has $\|d^s(f \circ \sigma)\| \leq A(k) + B(k)M(f)$ for some constants $A(k)$ and $B(k)$ that depend only on $k$. Assuming that $M(f) \geq 1$ this can be written as $\|d^s (f \circ \sigma)\|_{\infty} \leq c_1(k) M(f)$ for some constant $c_1(k) = A(k) + B(k)$ which depends only on $k$. When $M(f) < 1$ we can just write it as $\|d^s (f \circ \sigma)\|_{\infty} \leq c_1(k)$, but this case will be somewhat trivial in what follows.

(2) Let $J \subseteq [0,1]$ be a closed interval of length $\delta$ (which one thinks of as being short). Let $t_0$ be the midpoint of $J$ and denote by $P$ the Taylor polynomial of degree $k$ for $f \circ \sigma$ at $t_0$. Starting with the very crude trivial bound $\|(f \circ \sigma)^{(k)}(t) - (f \circ \sigma)^{(k)}(t_0) \| \leq 2 \|(f \circ \sigma)^{(k)}\|_{\infty}$ and integrating it one obtains \[\|(f \circ \sigma)^{(s)}|_J - P^{(s)}|_J\|_{\infty} \leq \frac{2 \|(f \circ \sigma)^{(k)}\|_{\infty}}{(k-s)!} \delta^{k-s} \leq 2c_1(k)M(f) \delta^{k-s}\] for $s = 0,1,\ldots,k$.  Suppose that $\psi : [0,1] \longrightarrow J$ is the standard affine reparameterization, so that $\psi' = \delta$. Clearly $(f \circ \sigma \circ \psi)^{(s)} = \delta^s (f \circ \sigma)^{(s)} \circ \psi$, and so  \begin{equation} \label{eq:cotaderivadas} \|(f \circ \sigma \circ \psi)^{(s)} - (P \circ \psi)^{(s)}\|_{\infty} = \delta^s \|(f \circ \sigma)^{(s)}|_J-P^{(s)}_J\|_{\infty} \leq 2c_1(k) M(f) \delta^k.\end{equation}
\medskip

{\it Proof of Lemma \ref{lem:noepsilon}.} We are finally ready to prove the one-step estimate. Let $B \subseteq V_{\lambda}$ be a ball of radius $1$ and denote by $B'$ the closed ball with the same center as $B$ but radius  $\nicefrac{3}{2}$. As in the statement of the lemma, set $S = \{t \in [0,1] : f \circ \sigma(t) \in B\}$. 

Partition $[0,1]$ as the union of $N \approx (4 c_1(k) M(f))^{\nicefrac{1}{k}}$ closed intervals $J_i$ of equal length $\delta = 1/N$, so that $2 c_1(k)M(f) \delta^k \leq \nicefrac{1}{2}$. Denote by $P_i$ the Taylor approximation to $(f \circ \sigma)|_{J_i}$ as described above and by $\psi_i : [0,1] \longrightarrow J_i$ the standard affine reparametrization.
Let us focus on one of the intervals $J_i$. By \eqref{eq:cotaderivadas} we have that $\|(f \circ \sigma \circ \psi_i) - (P_i \circ \psi_i)\|_{\infty} \leq \nicefrac{1}{2}$. This ensures that if $t \in S \cap J_i$ then $P_i(t)$ belongs to the ball $B'$. Therefore $S \cap J_i \subseteq \{t \in J_i : P_i(t) \in B'\} = \psi_i^{-1}(\{t \in [0,1] : P_i \circ \psi_i(t) \in B'\})$. By Lemma \ref{lem:poly} applied to $Q_i = P_i \circ \psi_i$ the set $\{t \in [0,1] : P_i \circ \psi_i(t) \in B'\}$ can be covered by at most $\alpha(k)$ intervals $J'_{ij}$ such that $(P_i \circ \psi_i)|_{J'_{ij}}$ is $1/2$--normalized. Taking the preimage of these via $\psi_i$ produces a family of at most $\alpha(k)$ intervals $J_{ij}$ which cover $S \cap J_i$. Therefore $S$ itself is covered by no more than $\alpha(k) N \approx \mu(k) M(f)^{\nicefrac{1}{k}}$ intervals $J_{ij}$, where $\mu(k)$ simply absorbs all factors other than $M(f)$.

To conclude the proof it only remains to show that each $f \circ \sigma|_{J_{ij}}$ is normalized. Let $\psi_{ij} : [0,1] \longrightarrow J_{ij}$ and $\psi'_{ij} : [0,1] \longrightarrow J'_{ij}$ be the standard affine reparameterizations. Observe that the definition of $J_{ij}$ as $\psi_i^{-1}(J'_{ij})$ ensures that $\psi_{ij} = \psi_i \circ \psi'_{ij}$. The condition that $(P_i \circ \psi_i)|_{J'_{ij}}$ be $1/2$--normalized thus reads $\|d^s(P_i \circ \psi_{ij})\| \leq \nicefrac{1}{2}$ for $s = 1,\ldots,k$. We also have \[\|d^s(f \circ \sigma \circ \psi_{ij}) - d^s(P_i \circ \psi_{ij})\| \leq \|d^s(f \circ \sigma \circ \psi_i) - d^s(P_i \circ \psi_i)\| \leq \nicefrac{1}{2}\] where the first inequality follows because $\psi'_{ij}$ has a derivative less than $1$ and the second from \eqref{eq:cotaderivadas} and the choice of $N$. Therefore $\|d^s(f \circ \sigma \circ \psi_{ij})\| \leq 1$ as required.

\bibliographystyle{plain}

\bibliography{biblio}

\begin{thebibliography}{10}

\bibitem{hecyo1}
H.~Barge and J.~J. S\'{a}nchez-Gabites.
\newblock Knots and solenoids that cannot be attractors of self-homeomorphisms
  of $\mathbb{R}^3$.
\newblock {\em Int. Math. Res. Not.}, 13:10373--10407, 2021.

\bibitem{hecyo3}
H.~Barge and J.J. S\'{a}nchez-Gabites.
\newblock Knotted toroidal sets, attractors and incompressible surfaces.
\newblock 2022.
\newblock Preprint. Provisionally available at
  \url{https://arxiv.org/abs/2212.14642}.

\bibitem{hecyo2}
H.~Barge and J.J. S\'{a}nchez-Gabites.
\newblock The geometric index and attractors of homeomorphisms of $\mathbb
  {R}^3$.
\newblock {\em Ergod. Theory Dyn. Syst.}, 43:50–77, 2023.

\bibitem{bing2}
R.~H. Bing.
\newblock Locally tame sets are tame.
\newblock {\em Ann. Math. (2)}, 59(1):145--158, 1954.

\bibitem{bowen1}
R.~Bowen.
\newblock Topological entropy and axiom {${\rm A}$}.
\newblock In {\em Global {A}nalysis ({P}roc. {S}ympos. {P}ure {M}ath., {V}ol.
  {XIV}, {B}erkeley, {C}alif., 1968)}, pages 23--41. Amer. Math. Soc.,
  Providence, R.I., 1970.

\bibitem{conley1}
C.~Conley.
\newblock {\em Isolated invariant sets and the {M}orse index}, volume~38 of
  {\em CBMS Regional Conference Series in Mathematics}.
\newblock American Mathematical Society, 1978.

\bibitem{gajuscho1}
R.M.~Schori D.J.~Garity, I.S.~Jubran.
\newblock A chaotic embedding of the {W}hitehead continuum.
\newblock {\em Houston J. Math.}, 23(1):33--44, 1997.

\bibitem{gajuscho2}
R.M.~Schori D.J.~Garity, I.S.~Jubran.
\newblock Corrigenda: ``{A} chaotic embedding of the {W}hitehead continuum''.
\newblock {\em Houston J. Math.}, 23(2):385--390, 1997.

\bibitem{chedwards1}
C.~H. Edwards.
\newblock {Concentric solid tori in the $3$--sphere}.
\newblock {\em Trans. Amer. Math. Soc.}, 102:1--17, 1962.

\bibitem{edwardskirby1}
R.D. Edwards and R.C. Kirby.
\newblock Deformations of spaces of imbeddings.
\newblock {\em Ann. of Math. (2)}, 93:63--88, 1971.

\bibitem{farbmargalit}
B.~Farb and D.~Margalit.
\newblock {\em A primer on mapping class groups}, volume~49 of {\em Princeton
  Mathematical Series}.
\newblock Princeton University Press, Princeton, NJ, 2012.

\bibitem{franksricheson}
J.~Franks and D.~Richeson.
\newblock Shift equivalence and the {C}onley index.
\newblock {\em Trans. Amer. Math. Soc.}, 352(7):3305--3322, 2000.

\bibitem{garay1}
B.~M. Garay.
\newblock Strong cellularity and global asymptotic stability.
\newblock {\em Fund. Math.}, 138:147--154, 1991.

\bibitem{graysonpugh1}
M.~Grayson and C.~Pugh.
\newblock Critical sets in 3-space.
\newblock {\em Int. Hautes \'{E}tudes Sci. Publ. Math.}, 77:5--61, 1993.

\bibitem{gromov}
M.~Gromov.
\newblock Entropy, homology and semialgebraic geometry.
\newblock {\em Ast\'{e}risque}, 145-146:225--240, 1987.

\bibitem{gunthersegal1}
B.~G{\"u}nther and J.~Segal.
\newblock Every attractor of a flow on a manifold has the shape of a finite
  polyhedron.
\newblock {\em Proc. Amer. Math. Soc.}, 119(1):321--329, 1993.

\bibitem{hatcher1}
A.E. Hatcher.
\newblock A proof of the {S}male conjecture, {${\rm Diff}(S^{3})\simeq {\rm
  O}(4)$}.
\newblock {\em Ann. of Math. (2)}, 117(3):553--607, 1983.

\bibitem{hirsch1}
M.W. Hirsch.
\newblock {\em Differential topology}.
\newblock Graduate Texts in Mathematics, No. 33. Springer-Verlag, New
  York-Heidelberg, 1976.

\bibitem{jutesis}
I.~Jubran.
\newblock {\em A chaotic embedding of the Whitehead continuum}.
\newblock PhD thesis, Oregon State University, 1992.
\newblock Available online at
  \url{https://ir.library.oregonstate.edu/downloads/1z40kw52f}.

\bibitem{kato1}
H.~Kato.
\newblock {Attractors in Euclidean spaces and shift maps on polyhedra}.
\newblock {\em Houston J. Math.}, 24:671--680, 1998.

\bibitem{katokhasselblatt}
A.~Katok and B.~Hasselblatt.
\newblock {\em Introduction to the modern theory of dynamical systems},
  volume~54 of {\em Encyclopedia of Mathematics and its Applications}.
\newblock Cambridge University Press, Cambridge, 1995.

\bibitem{manning1}
A.~Manning.
\newblock Topological entropy and the first homology group.
\newblock In {\em Dynamical systems---{W}arwick 1974 ({P}roc. {S}ympos. {A}ppl.
  {T}opology and {D}ynamical {S}ystems, {U}niv. {W}arwick, {C}oventry,
  1973/1974; presented to {E}. {C}. {Z}eeman on his fiftieth birthday)}, pages
  185--190. Lecture Notes in Math., Vol. 468, 1975.

\bibitem{misiurewiczszlenk}
M.~Misiurewicz and W.~Szlenk.
\newblock Entropy of piecewise monotone mappings.
\newblock {\em Studia Math.}, 67(1):45--63, 1980.

\bibitem{moise2}
E.~E. Moise.
\newblock {\em Geometric topology in dimensions 2 and 3}.
\newblock Springer-Verlag, 1977.

\bibitem{newhouse}
S.E. Newhouse.
\newblock Entropy and volume.
\newblock {\em Ergodic Theory Dynam. Systems}, 8$^*$(Charles Conley Memorial
  Issue):283--299, 1988.

\bibitem{nortonpugh1}
A.~Norton and C.~Pugh.
\newblock Critical sets in the plane.
\newblock {\em Michigan Math. J.}, 38(3):441--459, 1991.

\bibitem{pacoyo1}
F.~{R. Ruiz del Portal} and J.~J. S{\'a}nchez-Gabites.
\newblock \v{C}ech cohomology of attractors of discrete dynamical systems.
\newblock {\em J. Diff. Eq.}, 257(8):2826--2845, 2014.

\bibitem{robinson1}
C.~Robinson.
\newblock {\em Dynamical systems. Stability, symbolic dynamics, and chaos}.
\newblock Studies in Advanced Mathematics. CRC Press, Boca Raton, FL, 1999.

\bibitem{schduff}
A.C. Schaeffer and R.J. Duffin.
\newblock On some inequalities of {S}. {B}ernstein and {W}. {M}arkoff for
  derivatives of polynomials.
\newblock {\em Bull. Amer. Math. Soc.}, 44(4):289--297, 1938.

\bibitem{schubert}
H.~Schubert.
\newblock Knoten und {V}ollringe.
\newblock {\em Acta Math.}, 90:131--286, 1953.

\bibitem{yomdin}
Y.~Yomdin.
\newblock Volume growth and entropy.
\newblock {\em Israel J. Math.}, 57(3):285--300, 1987.

\end{thebibliography}

\end{document}